\documentclass[final,leqno,showlabe]{siamltex}
\usepackage{amsmath}
\usepackage{graphicx}
\usepackage{graphicx,epstopdf}
\usepackage{wrapfig}
\usepackage[caption=false]{subfig}
\usepackage{amssymb}
\usepackage{cases}

\usepackage{enumerate}
\usepackage{stmaryrd}
\usepackage{mathrsfs}
\usepackage{multirow}
\renewcommand{\theequation}{\arabic{section}.\arabic{equation}}
\usepackage[usenames]{color}
\usepackage{indentfirst}

\usepackage{array} 

\usepackage[bookmarksopen,bookmarksopenlevel=0,bookmarksdepth=2]{hyperref}
\usepackage{cleveref}

\usepackage{doi}

\numberwithin{equation}{section}

\SetSymbolFont{stmry}{bold}{U}{stmry}{m}{n}

\newtheorem{remark}{Remark}[section]

\newcommand{\be}{\begin{equation}}
\newcommand{\ee}{\end{equation}}

\title{A symmetrized parametric finite element method for
 anisotropic surface diffusion in 3D}

\author{Weizhu Bao\thanks{Department of Mathematics, National University of Singapore, Singapore, 119076 ({\it matbaowz@nus.edu.sg}). This author's research was supported by the Ministry of Education of Singapore grant MOE2019-T2-1-063 (R-146-000-296-112).}
\and Yifei Li\thanks{Department of Mathematics, National University of
Singapore, Singapore, 119076 ({\it e0444158@u.nus.edu}).}
}

\date{}

\begin{document}
\maketitle


\begin{abstract}
For the evolution of a closed surface under anisotropic surface diffusion  with a general anisotropic surface energy $\gamma(\boldsymbol{n})$ in three dimensions (3D),
where $\boldsymbol{n}$ is the  unit outward normal vector,
by introducing a novel symmetric positive definite surface energy matrix $\boldsymbol{Z}_k(\boldsymbol{n})$ depending on a stabilizing function $k(\boldsymbol{n})$ and the Cahn-Hoffman $\boldsymbol{\xi}$-vector,
we present a new symmetrized variational formulation for anisotropic surface diffusion with weakly or strongly anisotropic surface energy, which preserves two important structures including volume conservation and energy dissipation. Then we propose a structural-preserving parametric finite element method (SP-PFEM) to discretize the symmetrized variational problem, which preserves the volume in the discretized level. Under a relatively mild and simple condition on $\gamma(\boldsymbol{n})$,
we show that SP-PFEM  is unconditionally energy-stable for almost all anisotropic surface energies $\gamma(\boldsymbol{n})$ arising in practical applications. Extensive numerical results are reported to demonstrate
the efficiency and accuracy as well as energy dissipation of the proposed
SP-PFEM for solving anisotropic surface diffusion in 3D.
\end{abstract}




\pagestyle{myheadings} \markboth{W.~Bao, and Y.~Li}
{Symmetrized PFEM for anisotropic surface diffusion in 3D}

\section{Introduction}

In materials science and solid-state physics as well as many other
applications, surface energy is usually anisotropic due to lattice
orientational anisotropy  at material interfaces and/or surfaces \cite{taylor1992overview,giga2006surface}.
The anisotropic surface energy generates {\bf anisotropic surface diffusion}
-- an important and general process involving the motion of adatoms,
molecules and atomic clusters (adparticles) -- at materials surfaces and
interfaces in solids \cite{Cahn94}. The anisotropic surface diffusion is an
important kinetics and/or mechanism in surface phase formation
\cite{bekhtereva1988indium,chang1999thermodynamic}, epitaxial growth \cite{Fonseca14,gurtin2002interface}, heterogeneous catalysis \cite{hauffe1955application}, and other areas in
materials/surface science \cite{Thompson12}. It has significant and manifested
applications in solid-state physics and materials science as well as
computational geometry, such as solid-state dewetting \cite{jiang2012,Naffouti17,wang2015sharp,Thompson12,Ye10a}, crystal growth of nanomaterials \cite{barrett2012numerical}, evolution of voids in microelectronic circuits
\cite{li1999numerical,Suo97}, morphology development of alloys \cite{asaro1972interface}, quantum dots
manufacturing \cite{armelao2006recent}, deformation of images \cite{clarenz2000anisotropic}, etc.

As is shown in Figure \ref{fig: illustartion figure}, for a closed surface
$S:=S(t)$ in three dimensions (3D) associated with a given anisotropic
surface energy $\gamma(\boldsymbol{n})$, where $t\ge0$ is time and
$\boldsymbol{n}=(n_1,n_2,n_3)^T\in
{\mathbb S}^2$ represents the outward unit normal vector satisfying
$|\boldsymbol{n}|:=\sqrt{n_1^2+n_2^2+n_3^2}=1$,
the motion by anisotropic surface diffusion of the surface is described by the following
geometric flow
\cite{Mullins57,jiang2012,Naffouti17,wang2015sharp,Thompson12,Ye10a}:
\begin{equation}\label{eq: geo form}
	V_n=\Delta_S\, \mu,
\end{equation}
where $V_n$ denotes the normal velocity, $\Delta_S:= \nabla_S \cdot \nabla_S$ is the surface Laplace-Beltrami operator, $\nabla_S$ denotes the surface gradient with respect to the surface $S(t)$, and $\mu:=\mu(\boldsymbol{n})$ is the chemical potential (or weighted mean curvature denoted as $H_\gamma:=H_\gamma(\boldsymbol{n})$ in the literature) generated from
the surface energy functional $W(S):=W(S(t))=\int_{S(t)} \gamma(\boldsymbol{n})\, dS$ via
the thermodynamic variation as $\mu=\frac{\delta W(S)}{\delta S}=\lim_{\varepsilon\to0}\frac{W(S^\varepsilon)-W(S)}{\varepsilon}$
with $S^\varepsilon$ being a small perturbation of $S$ \cite{jiang2016solid,jiang2019sharp}. It is well-known that the evolution
of the surface $S(t)$ under the anisotropic surface diffusion
\eqref{eq: geo form} preserves the following two essential geometric
structures \cite{Cahn94}: (1) the volume of the region enclosed by the surface is conserved, and (2) the free surface energy (or weighted surface area)
$W(S)$ decreases in time. In fact, the motion governed by the anisotropic surface diffusion can be mathematically regarded as the $H^{-1}$-gradient flow of the free surface energy functional (or weighted surface area) $W(S)$
\cite{taylor1992ii}.
\begin{figure}[hbtp!]
\centering
\includegraphics[width=0.6\textwidth]{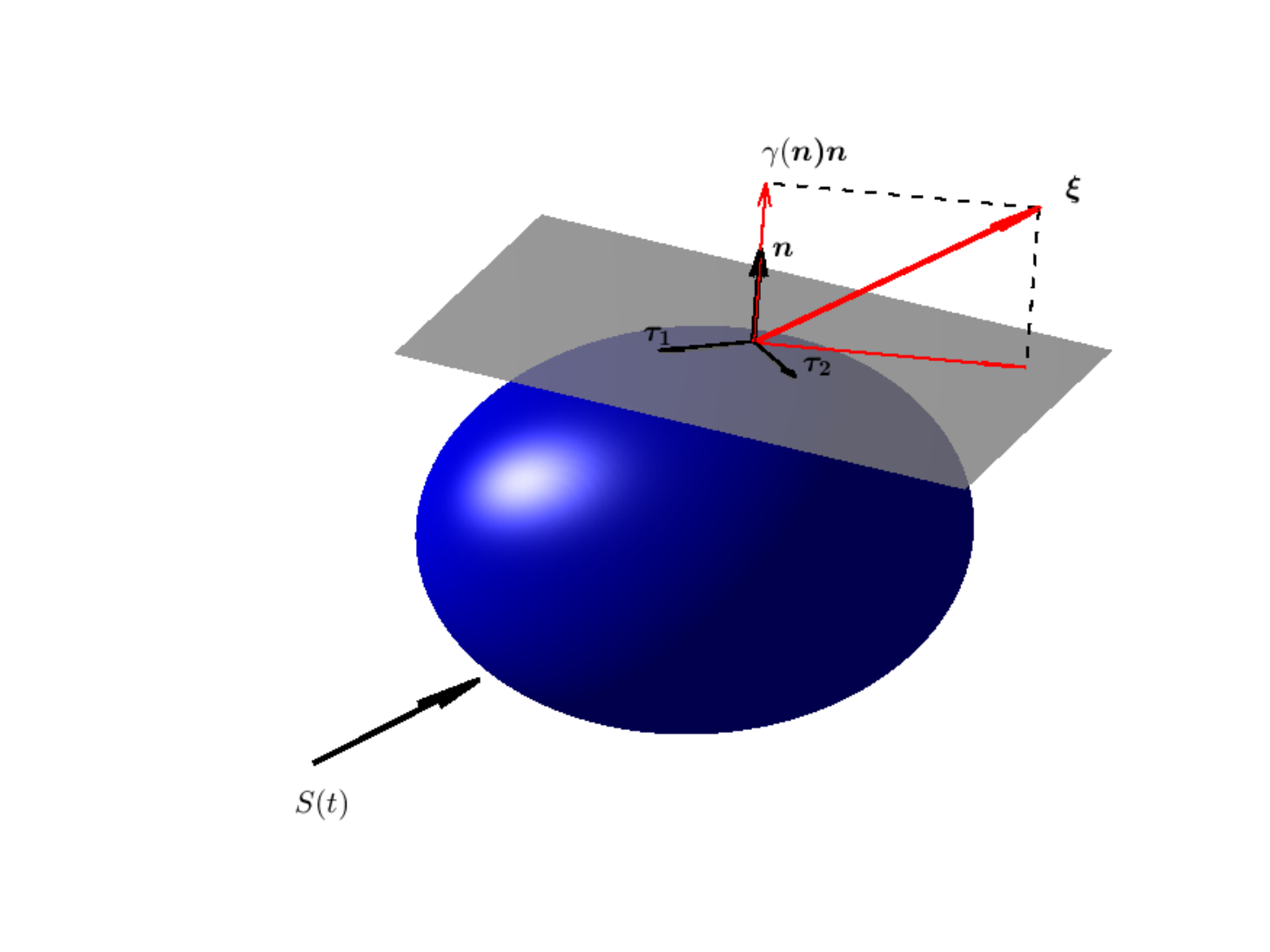}
\caption{An illustration of a closed surface $S(t)$ in $\mathbb{R}^3$
under anisotropic surface diffusion with an anisotropic surface energy $\gamma(\boldsymbol{n})$, where $\boldsymbol{n}$ is the outward unit normal vector, $\boldsymbol{\xi}$ is the Cahn-Hoffman $\boldsymbol{\xi}$-vector in \eqref{The xi-vector}, and $\boldsymbol{\tau}_1$ and $\boldsymbol{\tau}_2$ form a basis of the local tangential space.}
\label{fig: illustartion figure}
\end{figure}

Define $\gamma(\boldsymbol{p}):\ \mathbb{R}^3_*:=\mathbb{R}^3\setminus \{\boldsymbol{0}\}\to {\mathbb R}^+$ be a homogeneous extension of the anisotropic surface energy
$\gamma(\boldsymbol{n}):\ \mathbb{S}^2\to {\mathbb R}^+$ satisfying:
(i) $\gamma(c\boldsymbol{p})=c\gamma(\boldsymbol{p})$ for $c>0$ and $\boldsymbol{p}\in \mathbb{R}^3_*$, and  (ii) $\gamma(\boldsymbol{p})|_{\boldsymbol{p}=\boldsymbol{n}}
=\gamma(\boldsymbol{n})$ for $\boldsymbol{n}\in \mathbb{S}^2$. A typical homogeneous extension is widely used in the literature as \cite{jiang2019sharp,deckelnick2005computation}
\begin{equation}\label{gamma p}
\gamma(\boldsymbol{p}):=|\boldsymbol{p}|\gamma\left(\frac{\boldsymbol{p}}
{|\boldsymbol{p}|}\right),\qquad \forall \boldsymbol{p}=(p_1,p_2,p_3)^T\in \mathbb{R}^3_*:=\mathbb{R}^3\setminus \{\boldsymbol{0}\},
\end{equation}
where $|\boldsymbol{p}|=\sqrt{p_1^2+p_2^2+p_3^2}$.
Then the Cahn-Hoffman $\boldsymbol{\xi}$-vector introduced by Cahn and Hoffman and the Hessian matrix ${\bf H}_{\gamma}(\boldsymbol{n})$ of $\gamma(\boldsymbol{p})$ are mathematically given by \cite{hoffman1972vector,wheeler1999cahn}
\begin{equation}\label{The xi-vector}
\boldsymbol{\xi}=(\xi_1,\xi_2,\xi_3)^T:=\boldsymbol{\xi}(\boldsymbol{n})=\nabla \gamma(\boldsymbol{p})\big|_{\boldsymbol{p}}, \quad {\bf H}_{\gamma}(\boldsymbol{n}):=
\nabla\nabla\gamma(\boldsymbol{p})\big|_{\boldsymbol{p}=\boldsymbol{n}}, \quad \forall\boldsymbol{n}\in \mathbb{S}^2.
\end{equation}
Then the chemical potential $\mu$ (or weighted mean curvature) can be obtained as \cite{cahn1974vector}
\begin{equation}\label{mu-xi-relation}
\mu=\mu(\boldsymbol{n})=H_\gamma=H_\gamma(\boldsymbol{n})=\nabla_S \cdot\boldsymbol{\xi}=\nabla_S \cdot\boldsymbol{\xi}(\boldsymbol{n}), \qquad \forall\boldsymbol{n}\in \mathbb{S}^2.
\end{equation}
For any $\boldsymbol{n}\in \mathbb{S}^2$, we notice that ${\bf H}_{\gamma}(\boldsymbol{n})\boldsymbol{n}=\boldsymbol{0}$
and thus $0$ is an eigenvalue of ${\bf H}_{\gamma}(\boldsymbol{n})$ and $\boldsymbol{n}$ is a corresponding eigenvector. We denote the other two eigenvalues of ${\bf H}_{\gamma}(\boldsymbol{n})$ as $\lambda_1({\boldsymbol{n}})\le \lambda_2({\boldsymbol{n}}) \in{\mathbb R}$.
When $\gamma(\boldsymbol{n})\equiv {\rm constant}$ (e.g. $\gamma(\boldsymbol{n})\equiv 1$) for $\boldsymbol{n}\in \mathbb{S}^2$, i.e. with isotropic surface energy, then we have
$\gamma(\boldsymbol{p})=|\boldsymbol{p}|$ in \eqref{gamma p},
$\boldsymbol{\xi}=\boldsymbol{n}$ in \eqref{The xi-vector}, and
$\mu=H$ and ${\bf H}_{\gamma}(\boldsymbol{n})\equiv I_3-\boldsymbol{n}\boldsymbol{n}^T$ in \eqref{The xi-vector} with $H$ the mean curvature and $I_3$ the $3\times 3$ identity matrix and $\lambda_1(\boldsymbol{n})=\lambda_2(\boldsymbol{n})\equiv 1$,
and thus the anisotropic surface diffusion
\eqref{eq: geo form} collapses to the (isotropic) surface diffusion
with normal velocity given as $V_n=\Delta_S\, H$ \cite{barrett2007parametric,Mullins57}.
In contrast, when $\gamma(\boldsymbol{n})$ is not a constant, i.e.  with anisotropic surface energy:
when $\boldsymbol{\tau}^T{\bf H}_{\gamma}(\boldsymbol{n})\boldsymbol{\tau}>0$ for all $\boldsymbol{n},\boldsymbol{\tau}\in\mathbb{S}^2$ satisfying $\boldsymbol{\tau}\cdot \boldsymbol{n}:=\boldsymbol{\tau}^T \boldsymbol{n}=0$ ($\Leftrightarrow \lambda_2({\boldsymbol{n}})\ge\lambda_1({\boldsymbol{n}}) \geq0$ for all $\boldsymbol{n}\in\mathbb{S}^2$), it is called as {\sl weakly anisotropic}; and when $\boldsymbol{\tau}^T{\bf H}_{\gamma}(\boldsymbol{n})\boldsymbol{\tau}$ changes sign for $\boldsymbol{n},\boldsymbol{\tau}\in\mathbb{S}^2$ satisfying $\boldsymbol{\tau}\cdot \boldsymbol{n}=0$ ($\Leftrightarrow \lambda_1({\boldsymbol{n}})< 0$ for  some $\boldsymbol{n}\in\mathbb{S}^2$), it is called as {\sl strongly anisotropic}. For convenience of readers, we list several commonly-used anisotropic surface energies $\gamma(\boldsymbol{n})$ in the literature and their
corresponding Cahn-Hoffman $\boldsymbol{\xi}$-vectors in Appendix A.

Different numerical methods have been presented for solving the isotropic/anisotropic surface diffusion \eqref{eq: geo form}, such as the finite element method via graph evolution \cite{bansch2004surface,deckelnick2005computation}, the marker-particle method \cite{du2010tangent}, the discontinuous Galerkin finite element method \cite{xu2009local}, and the parametric finite element method (PFEM)
\cite{barrett2007parametric,barrett2008variational,Hausser07,bao2017parametric,jiang2019sharp,li2020energy}.
Among these methods, for isotropic surface diffusion,
the energy-stable PFEM (ES-PFEM) based on an elegant variational formulation, which was proposed by
Barrett, Garcke, and N{\"u}rnberg \cite{barrett2007parametric,barrett2008parametric}  (denoted as BGN scheme), performs the best in terms of efficiency and accuracy as well as mesh quality in practical computations, especially in two dimensions (2D).
The BGN scheme with unconditionally energy stability was successfully extended for solving solid-state dewetting problems with isotropic surface energy, i.e. motion of open curve and surface in 2D and 3D, respectively \cite{barrett2007parametric,barrett2008parametric}. It was also successfully extended for solving anisotropic surface diffusion Riemannian metric anisotropic surface energy \cite{barrett2008numerical,barrett2008variational}. Recently, based on BGN's variational formulation for surface diffusion, by approximating the normal vector in a clever way, Bao and Zhao \cite{bao2021structure,bao2022volume}
presented a structure-preserving PFEM (SP-PFEM) for surface diffusion,
which preserves area/volume conservation in 2D/3D and unconditionally energy dissiption in the discretized level.
Very recently, by introducing a proper surface energy matrix depending
on the Cahn-Hoffman $\boldsymbol{\xi}$-vector and a stabilizing function,
we obtained a new symmetrized (and conservative) variational formulation for anisotropic surface diffusion
with arbitrary surface energy in 2D and then designed structure-preserving and energy-stable PFEM under mild and simple conditions on the surface energy \cite{bao2021symmetrized}. The main aim of this paper is to extend the above method from 2D to 3D for anisotropic surface diffusion
with arbitrary surface energy. Again, the key is based on introducing a proper surface energy matrix depending
on the Cahn-Hoffman $\boldsymbol{\xi}$-vector and a stabilizing function in 3D and obtaining a new symmetrized (and conservative) variational formulation. The difficulty and major part is to establish unconditionally
energy dissipation of the full discretization under the following simple and mild condition on the arbitrary surface energy $\gamma(\boldsymbol{n})$ as
\begin{equation}\label{engstabgmp}
\gamma(-\boldsymbol{n})=\gamma(\boldsymbol{n}), \quad \forall \boldsymbol{n}\in \mathbb{S}^2, \qquad \gamma(\boldsymbol{p})\in C^2(\mathbb{R}^3\setminus \{\boldsymbol{0}\}).
\end{equation}

The paper is organized as follows. In section 2, we recall the mathematical representations for anisotropic surface diffusion, obtain a symmetrized variational formulation and propose a SP-FEM to discretize it. In section 3,
we establish energy stability of the proposed SP-PFEM. Extensive numerical results are reported to demonstrate the efficiency and accuracy as well as structure-preserving properties in section 4.
Finally, some conclusions are drawn in section 5.

\section{A new symmetrized variational formulation and its discretization}

This section first discusses the mathematical representations for the anisotropic surface diffusion. Then to introduce the weak formulation of $\mu$ in 3D, we generalize the surface energy matrix $\boldsymbol{Z}_k (\boldsymbol{n})$ for 2D anisotropic surface diffusion into 3D. A symmetrized conservative variational formulation for anisotropic surface diffusion in 3D is then derived by using the weak formulation and the surface energy matrix $\boldsymbol{Z}_k(\boldsymbol{n})$, and we show the two geometric properties are preserved for the new symmetrized variational formulation. Finally, by adopting backward Euler in time and the parametric finite element method in space, we derive the full discretization for the variational formulation and establish its structural preserving properties.

\subsection{Mathematical representations for anisotropic surface diffusion}
Let the closed surface $S:=S (t)$ be parameterized by $\boldsymbol{X} (\boldsymbol{\rho},t)$ as
\begin{equation}
	\boldsymbol{X}(t):\; \Omega\to\mathbb{R}^3, \boldsymbol{\rho}\,\mapsto \boldsymbol{X} (\boldsymbol{\rho},t)= (X_1 (\boldsymbol{\rho},t),X_2 (\boldsymbol{\rho},t),X_3 (\boldsymbol{\rho},t))^T,
\end{equation}
where $\Omega\subset \mathbb{R}^2$, then the motion of $S (t)$ under
the anisotropic surface diffusion \eqref{eq: geo form} can be mathematically
described by the following geometric partial differential equations via
the Cahn-Hoffman $\boldsymbol{\xi}$-vector as \cite{deckelnick2005computation}
\begin{subequations}
\label{eqn:original formulation gamma(n) 1}
\begin{numcases}{}
\label{eqn:original gamma(n) aniso eq1 1}
\partial_t\boldsymbol{X}(\boldsymbol{\rho},t) =(\Delta_S\, \mu) \boldsymbol{n}, \qquad \boldsymbol{\rho}\in\Omega, \quad t>0,  \\
\label{eqn:original gamma(n) aniso eq2 1}
\mu =\nabla_S \cdot\boldsymbol{\xi} , \qquad
\boldsymbol{\xi}=\nabla \gamma(\boldsymbol{p})\big|_{\boldsymbol{p}=\boldsymbol{n}}.
\end{numcases}
\end{subequations}
The anisotropic surface diffusion
\eqref{eqn:original formulation gamma(n) 1}
can also be regarded as a geometric flow from the given initial closed
surface $S_0:=S(0)\subset \mathbb{R}^3$ to the surface $S(t)\subset \mathbb{R}^3$. We define the function spaces over the evolving surface $S (t)=\boldsymbol{X} (\boldsymbol{\rho},t)$.
\begin{equation}\label{l2 space}
	L^2 (S (t)):=\Big\{u:S (t)\to \mathbb{R}\,\Big|\,\int_{S (t)}|u|^2 dA<\infty\Big\},
\end{equation}
equipped with the $L^2$-inner product
\begin{equation}\label{l2 inner product}
	\left (u,v\right)_{S (t)}:=\int_{S (t)}u\,v dA, \quad\forall u, v\in L^2 (S (t)),
\end{equation}
here $dA$ is the surface measure. This inner product can be extend to $[L^2 (S (t))]^3$ by replacing the scalar product $u\, v$ by the vector inner product $\boldsymbol{u}\cdot \boldsymbol{v}$. And we adopt the angle bracket to emphasize the inner product for two matrix-valued functions $\boldsymbol{U}, \boldsymbol{V}$ in $[L^2 (S (t))]^{3\times 3}$,
\begin{equation}\label{l2 inner product matrix}
 	\langle \boldsymbol{U}, \boldsymbol{V} \rangle_{S(t)}:=\int_{S(t)}\boldsymbol{U}:\boldsymbol{V}\,dA, \quad \forall \,\boldsymbol{U}, \boldsymbol{V}\in [L^2(S(t))]^{3\times 3},
 \end{equation}
 here $\boldsymbol{U}:\boldsymbol{V}=\text{Tr} (\boldsymbol{V}^T\boldsymbol{U})$ is the Frobenius inner product. Furthermore, we introduce the Sobolev space
\begin{equation}\label{h1 space}
	H^1 (S (t)):=\Big\{u:S (t)\to \mathbb{R}\,\Big|\, u\in L^2 (S (t)),\,\nabla_S u\in [L^2 (S (t))]^3\Big\}.
\end{equation}
And this definition is also valid for the function in $[H^1 (S (t))]^3$.

We refer to the definition of surface gradient $\nabla_S$ for scalar-valued functions \cite{deckelnick2005computation}. And the surface gradient for a vector-valued function $\boldsymbol{F}=(F_1, F_2, F_3)^T$ is defined as
\begin{equation}
	\nabla_S \boldsymbol{F}:=(\nabla_S F_1,\, \nabla_S F_2,\, \nabla_S F_3)^T\in \mathbb{R}^{3\times 3}.
\end{equation}

\subsection{A new symmetrized variational formulation and its property}

First, we generalize the 2D symmetric surface energy matrix $\boldsymbol{Z}_k (\boldsymbol{n})$ proposed in \cite{bao2021symmetrized} into 3D by
\begin{equation}\label{def of Z_k}
	\boldsymbol{Z}_k (\boldsymbol{n})=\gamma (\boldsymbol{n})I_3-\boldsymbol{n}\boldsymbol{\xi}^T (\boldsymbol{n})-\boldsymbol{\xi} (\boldsymbol{n})\boldsymbol{n}^T+k (\boldsymbol{n})\boldsymbol{n}\boldsymbol{n}^T, \quad\forall \boldsymbol{n}\in \mathbb{S}^2,
\end{equation}
where $I_3$ is the $3\times 3$ indentity matrix, $k (\boldsymbol{n})$ is the stabilizing function which ensures $\boldsymbol{Z}_k (\boldsymbol{n})$ is positive definite.

We then show this generalization of $\boldsymbol{Z}_k(\boldsymbol{n})$ is reasonable by showing the strong formulation $\mu\boldsymbol{n}=-\partial_s(\boldsymbol{Z}_k(\boldsymbol{n})\partial_s \boldsymbol{X})$ introduced in \cite{bao2021symmetrized} for the weighted mean curvature $\mu$ in 2D can be generalized to the following weak formulation in 3D.

\begin{lemma}[The weak formulation for $\mu$]\label{lem: weak mu}
The weighted mean curvature $\mu$ satisfies the following weak formulation.
\begin{equation}\label{alt def of mu}
	\left (\mu,\boldsymbol{n}\cdot \boldsymbol{\omega}\right)_{S}=\big\langle \boldsymbol{Z}_k (\boldsymbol{n})\nabla_S\boldsymbol{X}, \nabla_S \boldsymbol{\omega}\big\rangle_S,
\end{equation}
where $\boldsymbol{\omega}= (\omega_1,\omega_2,\omega_3)^T: S\to \mathbb{R}^3$ is a smooth test function.
\end{lemma}
\begin{proof}
We adopt the notation $\nabla_S f=(\underline{D}_1f, \underline{D}_2 f, \underline{D}_3 f)^T$. Noticing the fact $\underline{D}_i X_k=\delta_{i,k}-n_kn_i$ and $\nabla_S f\cdot \boldsymbol{n}=0$, we obtain
\begin{equation}\label{annoying underline}
	\nabla_S X_k\cdot\nabla_S \omega_l=\sum_{i=1}^3 (\delta_{i,k}-n_kn_i)\underline{D}_i\omega_l=\underline{D}_k\omega_l-n_k \nabla_S \omega_l\cdot \boldsymbol{n}=\underline{D}_k\omega_l.
\end{equation}

Substitute the identity \eqref{annoying underline} in \cite[equation (8.18)]{deckelnick2005computation} yields the following identity
\begin{equation}\label{thm2.1 id1}
	\left (\mu,\boldsymbol{n}\cdot \boldsymbol{\omega}\right)_{S}=\gamma (\boldsymbol{n})\sum_{l=1}^3\int_S \nabla_S X_l\cdot \nabla_S\omega_l \,dA-\sum_{k,l=1}^3\int_{S}\xi_k n_l\nabla_S X_k\cdot \nabla_S \omega_l \,dA.
\end{equation}
Obviously, the second term $\langle\gamma (\boldsymbol{n}) \nabla_S\boldsymbol{X},\nabla_S\boldsymbol{\omega}\rangle_S$ corresponds to $\gamma (\boldsymbol{n})I_3$ in $\boldsymbol{Z}_k (\boldsymbol{n})$. Now by simplifying the last term, we have
\begin{align}\label{thm2.1 id2}
\sum_{k,l=1}^3\int_{S}\xi_k n_l\nabla_S X_k\cdot \nabla_S \omega_l \,dA &=\int_S \left (\sum_{k=1}^3\xi_k (\nabla_S X_k)\right)\cdot\left (\sum_{l=1}^3n_l (\nabla_S\omega_l)\right)dA\nonumber\\
&=\int_S \left ( (\nabla_S\boldsymbol{X})^T\boldsymbol{\xi}\right)\cdot\left ( (\nabla_S\boldsymbol{\omega})^T\boldsymbol{n}\right)dA\nonumber\\
&=\int_S \text{Tr}\left ( (\nabla_S\boldsymbol{\omega})^T\boldsymbol{n}\boldsymbol{\xi}^T (\nabla_S\boldsymbol{X})\right)dA\nonumber\\
&=\int_S \left (\boldsymbol{n}\boldsymbol{\xi}^T (\nabla_S\boldsymbol{X})\right):\Bigl(\nabla_S\boldsymbol{\omega}\Bigr)dA\nonumber\\
&=\langle\boldsymbol{n}\boldsymbol{\xi}^T\nabla_S\boldsymbol{X},\nabla_S\boldsymbol{\omega}\rangle_S,
\end{align}
which is the $\boldsymbol{n}\boldsymbol{\xi}^T (\boldsymbol{n})$ part in $\boldsymbol{Z}_k (\boldsymbol{n})$.

Finally, recall the identity $\nabla_S \boldsymbol{X}=I_3 - \boldsymbol{n}\boldsymbol{n}^T$ and combine the two identities \eqref{thm2.1 id1} and \eqref{thm2.1 id2} yields
\begin{align}
	\left (\mu,\boldsymbol{n}\cdot \boldsymbol{\omega}\right)_{S}&=\langle (\gamma (\boldsymbol{n})I_3-\boldsymbol{n}\boldsymbol{\xi}^T)\nabla_S\boldsymbol{X}, \nabla_S \boldsymbol{\omega}\rangle_{S}\nonumber\\
	&=\langle \boldsymbol{Z}_k (\boldsymbol{n})\nabla_S\boldsymbol{X},\nabla_S \boldsymbol{\omega}\rangle_S+\langle (\boldsymbol{\xi}\boldsymbol{n}^T-k(\boldsymbol{n})\boldsymbol{n}\boldsymbol{n}^T)(I_3-\boldsymbol{n}\boldsymbol{n}^T), \nabla_S \boldsymbol{\omega}\rangle_{S}\nonumber\\
	&=\langle \boldsymbol{Z}_k(\boldsymbol{n})\nabla_S\boldsymbol{X},\nabla_S \boldsymbol{\omega}\rangle_S,
\end{align}
which is the desired result.
\end{proof}

With the weak formulation of $\mu$ \eqref{alt def of mu} given in lemma \ref{lem: weak mu}, by taking integration by parts, we can easily derive the following variational formulation for the anisotropic surface diffusion \eqref{eq: geo form}. For a given closed initial surface $S (0):=S_0$, find the solution $ (\boldsymbol{X} (\cdot,t), \mu (\cdot,t))\in [H^1(S(t))]^3\times H^1 (S (t))$ such that
\begin{subequations}
\label{eqn:New 3Daniso}
\begin{align}
\label{eqn:New 3Daniso1}
&\left (\partial_t\boldsymbol{X}\cdot\boldsymbol{n},\psi\right)_{S (t)}+\left (\nabla_S\mu,\nabla_S\psi\right)_{S (t)}=0\qquad \forall \psi\in H^1 (S (t)),\\[0.5em]
\label{eqn:New 3Daniso2}
&\left (\mu\boldsymbol{n},\boldsymbol{\omega}\right)_{S (t)}-\langle Z_k (\boldsymbol{n})\nabla_S\boldsymbol{X},\nabla_S \boldsymbol{\omega}\rangle_{S (t)}=0\qquad \forall \boldsymbol{\omega}\in [H^1 (S (t))]^3,
\end{align}
\end{subequations}

Denote the enclosed volume and the free energy of $S (t)$ as $V (t)$ and $W (t)$, respectively, which are defined by
\begin{equation}
	V (t):=\frac{1}{3}\int_{S (t)}\boldsymbol{X}\cdot\boldsymbol{n}\,dA,\qquad W (t):=\int_{S (t)}\gamma (\boldsymbol{n})\,dA.
\end{equation}
We then show the two geometric properties still hold for the variational formulation \eqref{eqn:New 3Daniso}.

\begin{theorem}
The enclosed volume $V (t)$ and the free energy $W (t)$ of the solution $S (t)$ of the variational formulation \eqref{eqn:New 3Daniso} are conserved and dissipative, respectively.
\end{theorem}
\begin{proof}
Taking the derivative of $V (t)$ with respect to $t$. From \cite{taylor1994linking}, we know that
\begin{equation}
	\frac{dV (t)}{dt}=\int_{S (t)}V_n\,dA=\left (\partial_t\boldsymbol{X}\cdot\boldsymbol{n},1\right)_{S (t)}=0.
\end{equation}

Similarly, the derivative of $W (t)$ with respect to $t$ is
\begin{equation}
	\frac{dW (t)}{dt}=\int_{S (t)} V_n\mu \,dA=\left (\partial_t\boldsymbol{X}\cdot\boldsymbol{n},\mu \right)_{S (t)}=-\left (\nabla_S\mu,\nabla_S\mu \right)_{S (t)},
\end{equation}
which means $\frac{dW (t)}{dt}\leq 0$.
\end{proof}


\subsection{A structural-preserving parametric finite element method}

We take $\tau>0$ to be the time step size, and the discrete time levels are $t_m=m\tau$ for each $m\geq 0$. For spatial discretization, as illustrated in figure \ref{fig: illuDir}, the surface $S (t_m)$ is approximated by a polyhedron $S^m=\cup_{j=1}^J\bar{\sigma}_j^m$ with $J$ mutually disjoint non-degenerated triangles surfaces $\sigma_j^m$ and $I$ vertices $\boldsymbol{q}_i^m$. We further denote $\{\boldsymbol{q}_{j_1}^m, \boldsymbol{q}_{j_2}^m, \boldsymbol{q}_{j_3}^m\}$ as the three counterclockwise vertices of the triangle $\sigma_j^m$, and $\mathcal{J}\{\sigma_j^m\}:= (\boldsymbol{q}_{j_2}^m-\boldsymbol{q}_{j_1}^m)\times (\boldsymbol{q}_{j_3}^m-\boldsymbol{q}_{j_2}^m)$ is the orientation vector with respect to $\sigma_j^m$, and the outward unit normal vector $\boldsymbol{n}_j^m$ of $\sigma_j^m$ is thus given by $\boldsymbol{n}_j^m=\frac{\mathcal{J}\{\sigma_j^m\}}{|\mathcal{J}\{\sigma_j^m\}|}$.

\begin{figure}[hbtp!]
\centering
\includegraphics[width=0.8\textwidth]{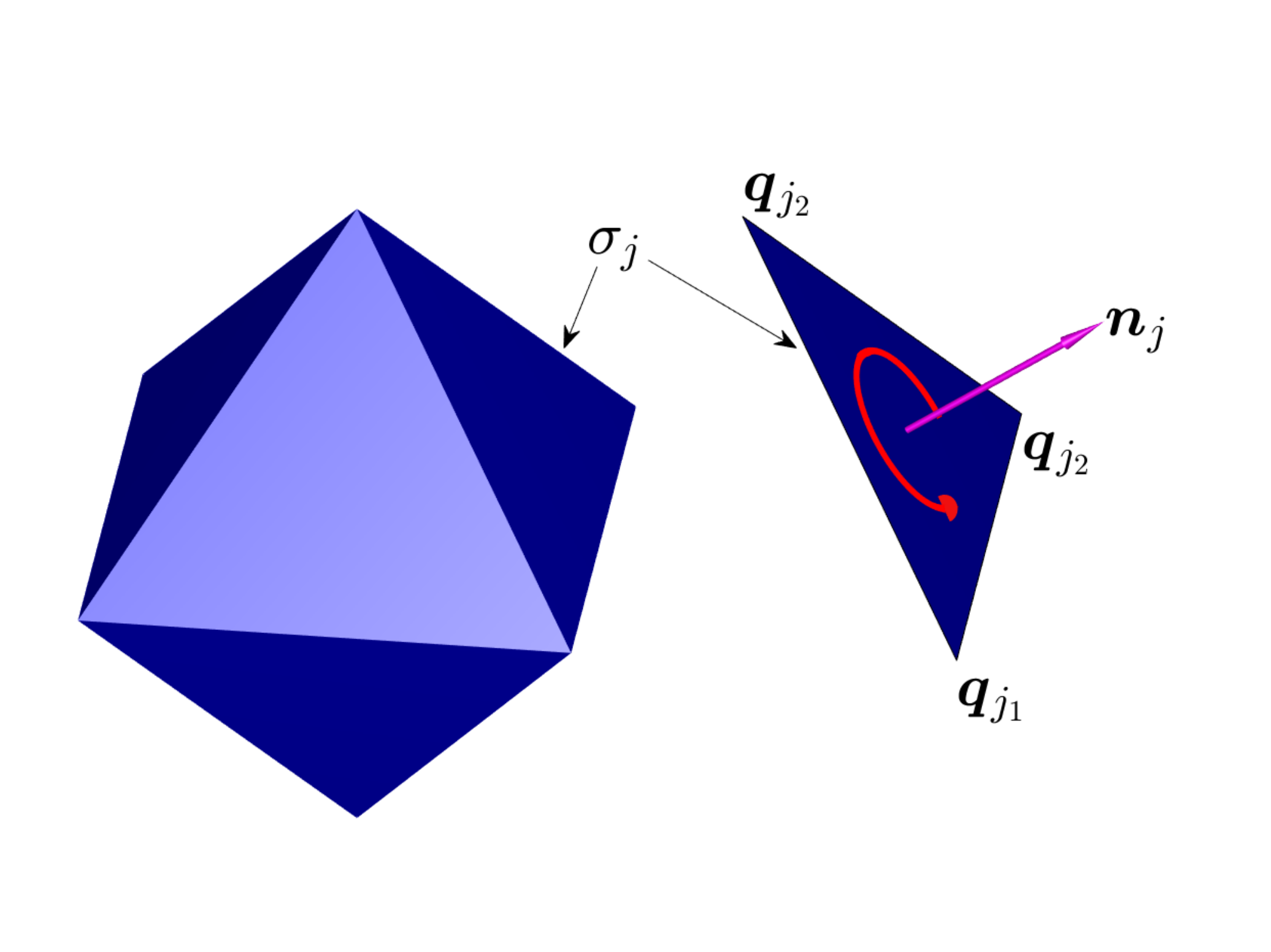}
\caption{An illustration of the approximation polyhedron $S^0$. The vertices $\{\boldsymbol{q}_{j_1}, \boldsymbol{q}_{j_2}, \boldsymbol{q}_{j_3}\}$ of the triangle $\sigma_j$ is oriented counterclockwise, see the red circular arrow. And the direction of the normal vector $\boldsymbol{n}_j$ is determined by the right-hand rule.}
\label{fig: illuDir}
\end{figure}

The finite element space with respect to the surface $S^m=\cup_{j=1}^J\bar{\sigma}_j^m$  is defined as follows
\begin{equation}
	\mathbb{K}^m:=\Big\{u\in C (S^m)\,\Big|\,u|_{\sigma_j^m}\in \mathcal{P}^1 (\sigma_j^m),\,\, \forall 1\leq j\leq J\Big\},
\end{equation}
equipped with the mass lumped inner product $\left (\cdot,\cdot\right)_{S^m}$ as
\begin{equation}
	\left (f,g\right)_{S^m}:=\frac{1}{3}\sum_{j=1}^J\sum_{i=1}^3|\sigma_j^m|f\left ( (\boldsymbol{q}_{j_i}^m)^-\right)g\left ( (\boldsymbol{q}_{j_i}^m)^-\right).
\end{equation}
Where $\mathcal{P}^1 (\sigma_j^m)$ is the space of polynomials on $\sigma_j^m$ with degree at almost $1$, $|\sigma_j^m|:=\frac{1}{2}|\mathcal{J}\{\sigma_j^m\}|$ denotes the area of $\sigma_j^m$, and $f\left ( (\boldsymbol{q}_{j_i}^m)^-\right)$ means the one-sided limit of $f (x)$ at $\boldsymbol{q}_{j_i}^m$ inside $\sigma_j^m$. This definition is also valid for vector- and matrix-valued function, and the inner product of the matrix-valued functions $\boldsymbol{U}$ and $\boldsymbol{V}$ is also emphsized by the angle bracket as
\begin{equation}
 	\langle \boldsymbol{U}, \boldsymbol{V}\rangle_{S^m}:=\frac{1}{3}\sum_{j=1}^J\sum_{i=1}^3 |\sigma_j^m| \boldsymbol{U}((\boldsymbol{q}_{j_i}^m)^-):\boldsymbol{V}((\boldsymbol{q}_{j_i}^m)^-).
 \end{equation}
Finally, the discretized surface gradient operator $\nabla_S$ for $f\in \mathbb{K}^m$ is given by
\begin{equation}\label{dis surface gradient}
	\nabla_S f|_{\sigma_j^m}:=\Bigl (f (\boldsymbol{q}_{j_1}^m) (\boldsymbol{q}_{j_2}^m-\boldsymbol{q}_{j_3}^m)+f (\boldsymbol{q}_{j_2}^m) (\boldsymbol{q}_{j_3}^m-\boldsymbol{q}_{j_1}^m)+f (\boldsymbol{q}_{j_3}^m) (\boldsymbol{q}_{j_1}^m-\boldsymbol{q}_{j_2}^m)\Bigr)\times \frac{\boldsymbol{n}_j^m}{|\mathcal{J}\{\sigma_j^m\}|},
\end{equation}
and for vector-valued function $\boldsymbol{F}= (F_1,F_2,F_3)^T\in [\mathbb{K}^m]^3, \nabla_S \boldsymbol{F}:= (\nabla_S F_1, \nabla_S F_2, \nabla_S F_3)^T$.

The full-implicit structural-preserving finite element method (SP-PFEM) for the variational formulation \eqref{eqn:New 3Daniso} can then be stated as follows: given the initial approximation $S^0=\cup_{j=1}^J\bar{\sigma}_j^0$ of $S (0)$. For each time step $t_m=m\tau, \,m=0,1,2,\ldots$, find the solution $\bigl (\boldsymbol{X}^{m+1}, \mu^{m+1}\bigr)\in [\mathbb{K}^m]^3\times \mathbb{K}^m$ such that
\begin{subequations}
\label{eqn:SP-PFEM}
\begin{align}
\label{eqn:SP-PFEM eq1}
&\left (\frac{\boldsymbol{X}^{m+1}-\boldsymbol{X}^m}{\tau}\cdot\boldsymbol{n}^{m+\frac{1}{2}},\psi\right)_{S^m}+\left (\nabla_S\mu^{m+1},\nabla_S\psi\right)_{S^m}=0,\qquad \forall \psi\in \mathbb{K}^m,\\[0.5em]
\label{eqn:SP-PFEM eq2}
&\left (\mu^{m+1}\boldsymbol{n}^{m+\frac{1}{2}},\boldsymbol{\omega}\right)_{S^m}-\langle \boldsymbol{Z}_k (\boldsymbol{n}^m)\nabla_S\boldsymbol{X}^{m+1},\nabla_S \boldsymbol{\omega}\rangle_{S^m}=0,\qquad \forall \boldsymbol{\omega}\in [\mathbb{K}^m]^3.
\end{align}
\end{subequations}
Here $\boldsymbol{X}^m (\boldsymbol{q}_i^m)=\boldsymbol{q}_i^m$, $\boldsymbol{X}^{m+1} (\boldsymbol{q}_i^m)=\boldsymbol{q}_i^{m+1}$ for each $i$, $\boldsymbol{n}^m|_{\sigma_j^m}=\boldsymbol{n}_j^m$,  $\sigma_j^{m+1}=\boldsymbol{X}^{m+1} (\sigma_j^m)$ is the triangle with counterclockwisely ordered vertices $\{\boldsymbol{q}_{j_1}^{m+1}, \boldsymbol{q}_{j_2}^{m+1}, \boldsymbol{q}_{j_3}^{m+1}\}$ for each $j$, and $S^{m+1}=\cup_{j=1}^J\bar{\sigma}_j^{m+1}$ for each $m$.
The semi-implicit approximation $\boldsymbol{n}^{m+\frac{1}{2}}$ of the outward normal vector $\boldsymbol{n}$ at $t=(m+\frac{1}{2})\tau$ is defined as follows
\begin{equation}\label{n half}
	\boldsymbol{n}^{m+\frac{1}{2}}|_{\sigma_{j}^m}:=\frac{\mathcal{J}\{\sigma_j^m\}+4\mathcal{J}\{\sigma_j^{m+\frac{1}{2}}\}+\mathcal{J}\{\sigma_j^{m+1}\}}{6|\mathcal{J}\{\sigma_j^m\}|},
\end{equation}
where $\sigma_j^{m+\frac{1}{2}}:=\frac{1}{2}\left (\sigma_j^m+\sigma_j^{m+1}\right)$.
\begin{remark}
We note the function $\boldsymbol{X}^{m+1}$ has different meanings at time step $t_m$ (as a function in $[\mathbb{K}^m]^3$) and $t_{m+1}$ (as a function in $[\mathbb{K}^{m+1}]^3$), and we adopt the same notation for simplicity.
\end{remark}

\subsection{Main results}
For the discretized polygon surface $S^m=\cup_{j=1}^J\bar{\sigma}_j^m$, its enclosed volume and surface energy are denoted as $V^m$ and $W^m$, respectively, which are defined as
\begin{subequations}
\begin{align}
\label{dis V}
&V^m:=\frac{1}{3}\int_{S^m}\boldsymbol{X}^m\cdot\boldsymbol{n}^mdA=\frac{1}{9}\sum_{j=1}^J\sum_{i=1}^3|\sigma_j^m|\boldsymbol{q}_{j_i}^m\cdot\boldsymbol{n}_j^m ,\\[0.5em]
\label{dis W}
&W^m:=\int_{S^m}\gamma (\boldsymbol{n}^m)dA=\sum_{j=1}^J|\sigma_j^m|\gamma (\boldsymbol{n}^m_j),\qquad\forall m\geq 0.
\end{align}
\end{subequations}
Denote the following axillary function $F_k (\boldsymbol{n}, \boldsymbol{u}, \boldsymbol{v}):[\mathbb{S}^2]^3\to \mathbb{R}$ as
\begin{align}\label{def of F_k}
	F_k (\boldsymbol{n}, \boldsymbol{u}, \boldsymbol{v}):= (\boldsymbol{u}^T\boldsymbol{Z}_k (\boldsymbol{n})\cdot\boldsymbol{u}) (\boldsymbol{v}^T\boldsymbol{Z}_k (\boldsymbol{n})\cdot\boldsymbol{v}),
\end{align}
and define the minimal stabilizing function $k_0 (\boldsymbol{n}):\mathbb{S}^2\to\mathbb{R}$ as (its existence will be given in next section)
\begin{equation}\label{def of k_0}
	k_0 (\boldsymbol{n})=\inf\Big\{k (\boldsymbol{n})\Big|F_k (\boldsymbol{n}, \boldsymbol{u},\boldsymbol{v})\geq \gamma^2 (\boldsymbol{u}\times\boldsymbol{v}),\quad \forall \boldsymbol{u},\boldsymbol{v}\in \mathbb{S}^2\Big\}.
\end{equation}
Then for the SP-PFEM \eqref{eqn:SP-PFEM}, we have 
\begin{theorem}[structural-preserving]\label{thm: main}
Assume $\gamma(\boldsymbol{n})$ satisfies \eqref{engstabgmp} and take
$k(\boldsymbol{n})$ in \eqref{def of Z_k} satisfying $k(\boldsymbol{n})\geq k_0(\boldsymbol{n})$ for $\boldsymbol{n}\in \mathbb{S}^2$, then the SP-PFEM \eqref{eqn:SP-PFEM} is volume conservation and energy dissiption, i.e.
\begin{subequations}
\begin{align}
\label{dis V prop}
&V^{m+1}=V^m=\ldots=V^0,\qquad \forall m\geq 0,\\[0.5em]
\label{dis W prop}
&W^{m+1}\leq W^m\leq\ldots\leq W^0,\qquad\forall m\geq 0.
\end{align}
\end{subequations}
\end{theorem}

The proof of volume conservation \eqref{dis V prop} is similar to \cite{bao2021structurepreserving} and thus it is omitted here for brevity,
and we will establish the energy dissipation or uncoditional energy stablity 
\eqref{dis W prop} in next section. 


\begin{remark}
The semi-discretization of the variational form \eqref{eqn:New 3Daniso} also preserves the two geometric properties. And the proof is similar to the isotropic case; we refer \cite{barrett2007parametric,Zhao,bao2021structurepreserving} for the result of semi-discretization of 3D isotropic surface diffusion.
\end{remark}

\section{Energy stability}
In this section, we first prove the existence of $k_0(\boldsymbol{n})$ and show its sub-linear property as a functional of $\gamma (\boldsymbol{n})$. By utilizing the existence of $k_0 (\boldsymbol{n})$ together with several lemmas, we finally prove the energy stability part of our main theorem \eqref{dis W prop}.
\subsection{Minimal stabilizing function}

From \eqref{def of k_0}, we know that $F_{k_0} (\boldsymbol{n}, \boldsymbol{u}, \boldsymbol{v})\geq 0$. Take $\boldsymbol{u}=\boldsymbol{n}$ in \eqref{def of F_k}, it yields $k_0 (\boldsymbol{n})-\gamma (\boldsymbol{n})\geq 0$, and we have a finite lower bound for $k_0 (\boldsymbol{n})$. To prove the existence of $k_0 (\boldsymbol{n})$, we only need to show $k_0 (\boldsymbol{n})<\infty$, and this is given by the following upper bound estimation of $k_0 (\boldsymbol{n})$. And we begin with the following lemma.
\begin{lemma}\label{lemma 1}
Let $G (\boldsymbol{n}, \boldsymbol{u}, \boldsymbol{v})$ be an auxillary function given by
\begin{equation}\label{auxillary function G}
	G (\boldsymbol{n}, \boldsymbol{u}, \boldsymbol{v}):=\gamma (\boldsymbol{n})\left (\gamma (\boldsymbol{n})-2 (\boldsymbol{\xi}\cdot \boldsymbol{u}) (\boldsymbol{n}\cdot\boldsymbol{u})-2 (\boldsymbol{\xi}\cdot\boldsymbol{v}) (\boldsymbol{n}\cdot\boldsymbol{v})\right),
\end{equation}
then for any $k (\boldsymbol{n})>0$, the following inequality holds
\begin{equation}\label{F-G}
	F_k (\boldsymbol{n}, \boldsymbol{u}, \boldsymbol{v})-G (\boldsymbol{n}, \boldsymbol{u}, \boldsymbol{v})\geq (\gamma (\boldsymbol{n})k (\boldsymbol{n})-4|\boldsymbol{\xi}|^2)\left ( (\boldsymbol{n}\cdot\boldsymbol{u})^2+ (\boldsymbol{n}\cdot\boldsymbol{v})^2\right).
\end{equation}
\end{lemma}
\begin{proof}By direct computation and the arithmetic-geometric mean inequality, we obtain
\begin{align*}
&F_{k} (\boldsymbol{n}, \boldsymbol{u}, \boldsymbol{v})-G (\boldsymbol{n}, \boldsymbol{u}, \boldsymbol{v})\\
&\geq\gamma (\boldsymbol{n})k (\boldsymbol{n})\left ( (\boldsymbol{n}\cdot\boldsymbol{u})^2+ (\boldsymbol{n}\cdot\boldsymbol{v})^2\right)+k (\boldsymbol{n})^2 (\boldsymbol{n}\cdot\boldsymbol{u})^2 (\boldsymbol{n}\cdot\boldsymbol{v})^2\\
&\quad-4|\boldsymbol{\xi}|^2| (\boldsymbol{n}\cdot\boldsymbol{u}) (\boldsymbol{n}\cdot\boldsymbol{v})|-2|\boldsymbol{\xi}|k (\boldsymbol{n})| (\boldsymbol{n}\cdot\boldsymbol{u}) (\boldsymbol{n}\cdot\boldsymbol{v})| (| (\boldsymbol{n}\cdot\boldsymbol{u})|+| (\boldsymbol{n}\cdot\boldsymbol{v})|)\\
&\geq (\gamma (\boldsymbol{n})k (\boldsymbol{n})-2|\boldsymbol{\xi}|^2)\left ( (\boldsymbol{n}\cdot\boldsymbol{u})^2+ (\boldsymbol{n}\cdot\boldsymbol{v})^2\right)+k (\boldsymbol{n})^2 (\boldsymbol{n}\cdot\boldsymbol{u})^2 (\boldsymbol{n}\cdot\boldsymbol{v})^2\\
&\quad -k (\boldsymbol{n})|\boldsymbol{\xi}|\left ( (\boldsymbol{n}\cdot\boldsymbol{u})^2 (\frac{2|\boldsymbol{\xi}|}{k (\boldsymbol{n})}+\frac{k (\boldsymbol{n})}{2|\boldsymbol{\xi}|} (\boldsymbol{n}\cdot\boldsymbol{v})^2)+ (\boldsymbol{n}\cdot\boldsymbol{v})^2 (\frac{2|\boldsymbol{\xi}|}{k (\boldsymbol{n})}+\frac{k (\boldsymbol{n})}{2|\boldsymbol{\xi}|} (\boldsymbol{n}\cdot\boldsymbol{u})^2)\right)\\
&= (\gamma (\boldsymbol{n})k (\boldsymbol{n})-4|\boldsymbol{\xi}|^2)\left ( (\boldsymbol{n}\cdot\boldsymbol{u})^2+ (\boldsymbol{n}\cdot\boldsymbol{v})^2\right),
\end{align*}
which is the desired inequality.
\end{proof}

Since $\gamma (\boldsymbol{n})$ is not differentiable at $\boldsymbol{0}$, the cross product $\gamma^2 (\boldsymbol{u}\times \boldsymbol{v})\not\in C^2 (\mathbb{S}^2\times \mathbb{S}^2)$. The following lemma is helpful in estimating $\gamma^2 (\boldsymbol{u}\times \boldsymbol{v})$,
\begin{lemma}
	 For any $\gamma(\boldsymbol{n})$ satisfying \eqref{engstabgmp}, i.e., $\gamma (\boldsymbol{n})\in C^2 (\mathbb{S}^2)$ with $\gamma (-\boldsymbol{n})=\gamma (\boldsymbol{n})$, $\gamma^2 (\boldsymbol{p})$ is continuous differentiable in $\mathbb{R}^3$. Moreover, there exists a constant $C_1$ defined by
	 \begin{equation}
	 	C_1=\frac{1}{2}\sup\limits_{\boldsymbol{n}\in\mathbb{S}^2}\left\|{\bf H}_{\gamma^2} (\boldsymbol{n})\right\|_2,
	 \end{equation}such that
	 \begin{equation}\label{est of cross prod}
	 	\gamma^2 (\boldsymbol{p})-\gamma^2 (\boldsymbol{q})\leq \nabla (\gamma^2 (\boldsymbol{q}))\cdot (\boldsymbol{p}-\boldsymbol{q})+C_1|\boldsymbol{p}-\boldsymbol{q}|^2,\qquad\forall \boldsymbol{p},\boldsymbol{q}\in\mathbb{R}^3.
	 \end{equation}
	 Where $\left\|\cdot\right\|_2$ is the spectural norm.
	 \end{lemma}
	 \begin{proof}
	It's straitforward to check $\gamma^2 (\boldsymbol{p})$ is continuous differentiable by definition. To prove the inequality \eqref{est of cross prod}, we first consider the case the line segment of $\boldsymbol{p}, \boldsymbol{q}$ does not pass $\boldsymbol{0}$, i.e., $\lambda \boldsymbol{p}+ (1-\lambda)\boldsymbol{q}\neq\boldsymbol{0}, \forall\, 0\leq\lambda\leq 1$. Since $\gamma^2 (\boldsymbol{p})$ is homogeneous of degree $2$, we know that ${\bf H}_{\gamma^2} (\boldsymbol{p})$ is homogeneous of degree $0$, which yields
	 \begin{equation}
	 	{\bf H}_{\gamma^2}(\boldsymbol{\zeta})={\bf H}_{\gamma^2}\left (\boldsymbol{\zeta}/|\boldsymbol{\zeta}|\right),\qquad\forall \boldsymbol{\zeta}\neq\boldsymbol{0}.
	 \end{equation}
	 By mean value theorem, there exists a $\lambda_0\in (0,1)$ and $\boldsymbol{\zeta}=\lambda_0\boldsymbol{p}+ (1-\lambda_0)\boldsymbol{q}\neq \boldsymbol{0}$, such that
	 \begin{equation}
	 	\gamma^2 (\boldsymbol{p})=\gamma^2 (\boldsymbol{q})+\nabla (\gamma^2 (\boldsymbol{q}))\cdot (\boldsymbol{p}-\boldsymbol{q})+\frac{1}{2} (\boldsymbol{p}-\boldsymbol{q})^T{\bf H}_{\gamma^2}(\boldsymbol{\zeta})\cdot (\boldsymbol{p}-\boldsymbol{q}).
	 \end{equation}
	 Thus \eqref{est of cross prod} holds for such $\boldsymbol{p}, \boldsymbol{q}$.

	  If $\boldsymbol{0}$ is contained in line segment of $\boldsymbol{p},\boldsymbol{q}$, we can find a sequence $ (\boldsymbol{p}_k, \boldsymbol{q}_k)\to (\boldsymbol{p}, \boldsymbol{q})$ such that for each $k$, the line segment of $\boldsymbol{p}_k,\boldsymbol{q}_k$ does not pass $\boldsymbol{0}$. We know \eqref{est of cross prod} holds for such $\boldsymbol{p}_k, \boldsymbol{q}_k$. And the \eqref{est of cross prod} is thus valid for all $\boldsymbol{p},\boldsymbol{q}$ by the continuity of $\gamma^2 (\boldsymbol{p})$ and $\nabla (\gamma^2 (\boldsymbol{p}))$.
	 \end{proof}
\begin{theorem}\label{thm: K}Suppose $\gamma(\boldsymbol{n})$ satisfies the energy stable condition \eqref{engstabgmp}. Then there exists a constant $K (\boldsymbol{n})<\infty$ only depends on $\gamma (\boldsymbol{n})$ given by
\begin{equation}\label{def of K}
 	K (\boldsymbol{n})=\frac{6|\boldsymbol{\xi}|^2+8\gamma (\boldsymbol{n})|\boldsymbol{\xi}|+16C_1}{\gamma (\boldsymbol{n})}<\infty,
 \end{equation} such that $F_K (\boldsymbol{n}, \boldsymbol{u}, \boldsymbol{v})\geq \gamma^2 (\boldsymbol{u}\times\boldsymbol{v}), \,\forall \boldsymbol{u}, \boldsymbol{v}\in \mathbb{S}^2$.
\end{theorem}
\begin{proof}
It is convenient to consider the special case $\boldsymbol{n}= (0, 0, 1)^T$. We write $\boldsymbol{u}, \boldsymbol{v}$ in the spherical coordinates,
\begin{subequations}
\label{eqn:sphi coord}
\begin{align}
\label{eqn:sphi coord u}
&\boldsymbol{u}= (\cos\phi_1\cos\theta_1, \cos\phi_1\sin\theta_1, \sin\phi_1)^T,\\[0.5em]
\label{eqn:sphi coord v}
&\boldsymbol{v}= (\cos\phi_2\cos\theta_2, \cos\phi_2\sin\theta_2, \sin\phi_2)^T,
\end{align}
\end{subequations}
where $-\frac{\pi}{2}\leq \phi_1,\phi_2\leq \frac{\pi}{2},\, 0\leq \theta_1, \theta_2<2\pi$. And in case $\phi=\pm \frac{\pi}{2}$, we choose $\theta=0$. The cross product $\boldsymbol{u}\times \boldsymbol{v}$ is then represented as
\begin{equation}
 	\label{cross of u and v}
	\boldsymbol{u}\times \boldsymbol{v}=\cos\phi_2\sin\phi_1\hat{\boldsymbol{v}}_0+\cos\phi_1\sin\phi_2\hat{\boldsymbol{u}}_0+\cos\phi_1\cos\phi_2(0,0,\sin\theta_{2,1})^T,
 \end{equation}
 where
 \begin{equation}
 	\hat{\boldsymbol{u}}_0=(\sin\theta_1, -\cos\theta_1, 0)^T,\,\,\hat{\boldsymbol{v}}_0=(-\sin\theta_2,\cos\theta_2,0)^T,\,\, \theta_{2,1}=\theta_2-\theta_1.
 \end{equation}
And denote $\boldsymbol{u}_0, \boldsymbol{v}_0$ as
\begin{equation}
\label{eqn:sphi coord2}
\boldsymbol{u}_0:= (\cos\theta_1, \sin\theta_1, 0)^T,\quad \boldsymbol{v}_0:= (\cos\theta_2, \sin\theta_2, 0)^T.
\end{equation}

Since $|\boldsymbol{u}|, |\boldsymbol{v}|, |\boldsymbol{u}_0|, |\boldsymbol{v}_0|=1$, we know that $| (\boldsymbol{u}-\boldsymbol{u}_0)\times \boldsymbol{v}|\leq |\boldsymbol{u}-\boldsymbol{u}_0|, |\boldsymbol{u}\times (\boldsymbol{v}-\boldsymbol{v}_0)|\leq |\boldsymbol{v}-\boldsymbol{v}_0|$, $| (\boldsymbol{u}-\boldsymbol{u}_0)\times (\boldsymbol{v}-\boldsymbol{v}_0)|\leq |\boldsymbol{u}-\boldsymbol{u}_0|+|\boldsymbol{v}-\boldsymbol{v}_0|$, and thus
\begin{equation}
 	|\boldsymbol{u}\times \boldsymbol{v}-\boldsymbol{u}_0\times\boldsymbol{v}_0|^2\leq 8 (|\boldsymbol{u}-\boldsymbol{u}_0|^2+|\boldsymbol{v}-\boldsymbol{v}_0|^2).
 \end{equation} By taking $\boldsymbol{p}=\boldsymbol{u}\times\boldsymbol{v}, \boldsymbol{q}=\boldsymbol{u}_0\times \boldsymbol{v}_0$ in \eqref{est of cross prod}, and noticing $\boldsymbol{u}_0\times \boldsymbol{v}_0=(\sin\theta_{2,1})\, \boldsymbol{n}$, we obtain
\begin{align}\label{ineq: est for cross prod}
&\gamma^2 (\boldsymbol{u}\times\boldsymbol{v})- (\sin\Delta \theta)^2\gamma^2 (\boldsymbol{n})\nonumber\\
&\leq\sin\theta_{2,1}\nabla (\gamma^2 ( \boldsymbol{n}))\cdot (\boldsymbol{u}\times\boldsymbol{v}-\boldsymbol{u}_0\times \boldsymbol{v}_0)+C_1|\boldsymbol{u}\times\boldsymbol{v}-\boldsymbol{u}_0\times \boldsymbol{v}_0|^2\nonumber\\
&\leq 2\gamma (\boldsymbol{n})\boldsymbol{\xi}\cdot\left (\sin\phi_1\,\hat{\boldsymbol{v}}_0+\sin\phi_2\,\hat{\boldsymbol{u}}_0\right)\sin\theta_{2,1}\nonumber\\
&\qquad +4\gamma (\boldsymbol{n})|\boldsymbol{\xi}| ( (\sin\phi_1)^2+ (\sin\phi_2)^2)+8C_1 (|\boldsymbol{u}-\boldsymbol{u}_0|^2+|\boldsymbol{v}-\boldsymbol{v}_0|^2)\nonumber\\
&\leq 2\gamma (\boldsymbol{n})\boldsymbol{\xi}\cdot\left ( (\cos\theta_{2,1}\, \boldsymbol{v}_0-\boldsymbol{u}_0)\, (\boldsymbol{n}\cdot\boldsymbol{u})+ (\cos\theta_{2,1}\, \boldsymbol{u}_0-\boldsymbol{v}_0)\, (\boldsymbol{n}\cdot\boldsymbol{v})\right)\\
&\qquad +4 (\gamma (\boldsymbol{n})|\boldsymbol{\xi}|+4C_1) ( (\boldsymbol{n}\cdot\boldsymbol{u})^2+ (\boldsymbol{n}\cdot\boldsymbol{v})^2).\nonumber
\end{align}
Where we use the facts $|\boldsymbol{u}-\boldsymbol{u}_0|=2|\sin\frac{\phi_1}{2}|, |\boldsymbol{v}-\boldsymbol{v}_0|=2|\sin\frac{\phi_2}{2}|$ and $ (\sin\phi)^2\geq 2 (\sin\frac{\phi}{2})^2=1-\cos\phi ,\,\, \forall -\frac{\pi}{2}\leq \phi\leq \frac{\pi}{2}$.

To estimate $G (\boldsymbol{n}, \boldsymbol{u}, \boldsymbol{v})$, we observe the following inequalities
\begin{subequations}
\label{eqn:est for G}
\begin{align}
 (\boldsymbol{\xi}\cdot\boldsymbol{u}) (\boldsymbol{n}\cdot\boldsymbol{u})&= (\boldsymbol{\xi}\cdot\boldsymbol{u}_0) (\boldsymbol{n}\cdot\boldsymbol{u})+ (\boldsymbol{\xi}\cdot (\boldsymbol{u}-\boldsymbol{u}_0)) (\boldsymbol{n}\cdot (\boldsymbol{u}-\boldsymbol{u}_0))\nonumber\\
&\geq (\boldsymbol{\xi}\cdot\boldsymbol{u}_0) (\boldsymbol{n}\cdot\boldsymbol{u})-2|\boldsymbol{\xi}| (\boldsymbol{n}\cdot\boldsymbol{u})^2,\\[1em]
\label{eqn:est for G2}
 (\boldsymbol{\xi}\cdot\boldsymbol{v}) (\boldsymbol{n}\cdot\boldsymbol{v})&\geq (\boldsymbol{\xi}\cdot\boldsymbol{v}_0) (\boldsymbol{n}\cdot\boldsymbol{v})-2|\boldsymbol{\xi}| (\boldsymbol{n}\cdot\boldsymbol{v})^2.
\end{align}
\end{subequations}
Combining \eqref{auxillary function G} and \eqref{eqn:est for G} yields
\begin{equation}\label{est of G}
	G (\boldsymbol{n}, \boldsymbol{u}, \boldsymbol{v})\geq G (\boldsymbol{n}, \boldsymbol{u}_0, \boldsymbol{v}_0)-4\gamma (\boldsymbol{n})|\boldsymbol{\xi}| ( (\boldsymbol{n}\cdot\boldsymbol{u})^2+ (\boldsymbol{n}\cdot\boldsymbol{v})^2).
\end{equation}

Finally, by \eqref{F-G} in lemma \ref{lemma 1}, the estimation of $\gamma^2 (\boldsymbol{u}\times \boldsymbol{v})$ \eqref{ineq: est for cross prod}, and the estimation of $G (\boldsymbol{n}, \boldsymbol{u}, \boldsymbol{v})$ \eqref{est of G}, we obtain
\begin{align*}
&F_K (\boldsymbol{n}, \boldsymbol{u}, \boldsymbol{v})-\gamma^2 (\boldsymbol{u}\times \boldsymbol{v})\\
&\geq \gamma (\boldsymbol{n})^2 (\cos\theta_{2,1})^2-2\gamma (\boldsymbol{n})\cos\theta_{2,1}\,\left ( (\boldsymbol{\xi}\cdot\boldsymbol{v}_0) (\boldsymbol{n}\cdot\boldsymbol{u})+ (\boldsymbol{\xi}\cdot\boldsymbol{u}_0) (\boldsymbol{n}\cdot\boldsymbol{v})\right)\\
&\qquad + (\gamma (\boldsymbol{n})K (\boldsymbol{n})-4|\boldsymbol{\xi}|^2-8\gamma (\boldsymbol{n})|\boldsymbol{\xi}|-16C_1) ( (\boldsymbol{n}\cdot\boldsymbol{u})^2+ (\boldsymbol{n}\cdot\boldsymbol{v})^2)\\
&\geq \gamma (\boldsymbol{n})^2 (\cos\theta_{2,1})^2-2\gamma (\boldsymbol{n})|\cos\theta_{2,1}|\,|\boldsymbol{\xi}|\left (  (\boldsymbol{n}\cdot\boldsymbol{u})+ (\boldsymbol{n}\cdot\boldsymbol{v})\right)\\
&\qquad + 2|\boldsymbol{\xi}|^2( (\boldsymbol{n}\cdot\boldsymbol{u})^2+ (\boldsymbol{n}\cdot\boldsymbol{v})^2)\\
&\geq 0.
\end{align*}
Thus we have $F_K(\boldsymbol{n}, \boldsymbol{u}, \boldsymbol{v})\geq \gamma^2 (\boldsymbol{u}\times \boldsymbol{v})$ for $\boldsymbol{n}= (0, 0, 1)^T$. And since the constant $K (\boldsymbol{n})$ only depends on $\gamma (\boldsymbol{n})$, the proof is valid for arbitrary $\boldsymbol{n}\in \mathbb{S}^2$.
\end{proof}

Theorem \ref{thm: K} indicates that the set $\Big\{k (\boldsymbol{n})\Big|F_k (\boldsymbol{n}, \boldsymbol{u},\boldsymbol{v})\geq \gamma^2 (\boldsymbol{u}\times\boldsymbol{v}),\quad \forall \boldsymbol{u},\boldsymbol{v}\in \mathbb{S}^2\Big\}$ contains an element $K (\boldsymbol{n})<\infty$, and thus is not empty. Together with the fact $k_0 (\boldsymbol{n})\geq \gamma (\boldsymbol{n})$ yields the existence of $k_0 (\boldsymbol{n})$.

\begin{corollary}[existence of the minimal stabilizing function]
Suppose $\gamma (\boldsymbol{n})\in C^2$ with $\gamma (\boldsymbol{n})=\gamma (-\boldsymbol{n})$. Then the minimal stabilizing function $k_0 (\boldsymbol{n})$ in \eqref{def of k_0} is well-defined.
\end{corollary}

Finally, we point out the minimal stabilizing function $k_0 (\boldsymbol{n})$ is determined by $\gamma (\boldsymbol{n})$. And similar to the 2D result in \cite{bao2021symmetrized}, this map is sub-linear.
\begin{theorem}[positive homogeneity and subadditivity]Suppose $\gamma (\boldsymbol{n})=\gamma_1 (\boldsymbol{n})+\gamma_2 (\boldsymbol{n})$. Let $k_0 (\boldsymbol{n}), k_1 (\boldsymbol{n}), k_2 (\boldsymbol{n})$ be the minimal stabilizing function of $\gamma (\boldsymbol{n}), \gamma_1 (\boldsymbol{n}), \gamma_2 (\boldsymbol{n})$, respectively, we know that
\begin{itemize}
	\item $\forall c>0$, $ck_0 (\boldsymbol{n})$ is the stabilizing function of $c\gamma (\boldsymbol{n})$;
	\item $k_0 (\boldsymbol{n})\leq k_1 (\boldsymbol{n})+k_2 (\boldsymbol{n})$.
\end{itemize}
\end{theorem}
\begin{proof}
The proof of positive homogeneity is similar to the proof of lemma 4.4 in \cite{bao2021symmetrized}, thus we only prove subadditivity. Here we use $\boldsymbol{\xi}, \boldsymbol{\xi}_1, \boldsymbol{\xi}_2$ to denote the $\boldsymbol{\xi}$ vector for $\gamma (\boldsymbol{n}), \gamma_1 (\boldsymbol{n}), \gamma_2 (\boldsymbol{n})$, respectively.

Since $k_1 (\boldsymbol{n})$ is the minimal stabilizing function of $\gamma_1 (\boldsymbol{n})$. For any $t\in \mathbb{R}$, we have
\begin{align}
&\frac{1}{2}\boldsymbol{u}^T\boldsymbol{Z}_{k_1} (\boldsymbol{n})\boldsymbol{u}+\frac{t^2}{2}\boldsymbol{v}^T\boldsymbol{Z}_{k_1}\boldsymbol{v}-t\gamma_1 (\boldsymbol{u}\times\boldsymbol{v})\nonumber\\
&\geq 2\sqrt{\frac{t^2}{4}F_{k_1} (\boldsymbol{n}, \boldsymbol{u}, \boldsymbol{v})}-t\gamma_1 (\boldsymbol{u}\times\boldsymbol{v})\nonumber\\
&\geq 0.
\end{align}
And this inequality is also true for $\gamma_2 (\boldsymbol{n})$. Add the two inequalities together, and noticing $\boldsymbol{\xi}=\boldsymbol{\xi_1}+\boldsymbol{\xi_2}$, it yields that
\begin{align}
\frac{1}{2}\boldsymbol{u}^T\boldsymbol{Z}_{k_1+k_2} (\boldsymbol{n})\boldsymbol{u}+\frac{t^2}{2}\boldsymbol{v}^T\boldsymbol{Z}_{k_1+k_2}\boldsymbol{v}-t\gamma (\boldsymbol{u}\times\boldsymbol{v})\geq 0, \qquad\forall t\in \mathbb{R},
\end{align}
which means its discriminant $\gamma^2 (\boldsymbol{u}\times \boldsymbol{v})-F_{k_1+k_2} (\boldsymbol{n}, \boldsymbol{u}, \boldsymbol{v})\leq 0$. And the subadditivity is a direct conclusion from the definition of minimal stabilizing function \eqref{def of k_0}.
\end{proof}
\subsection{Proof of the main theorem}
By establishing the existence of $k_0 (\boldsymbol{n})$, we now have enough tools to prove \eqref{dis W prop} in theorem \ref{thm: main}. To simplify the proof, we first introduce the following alternative definition for the surface gradient operator $\nabla_S$.

\begin{lemma}Suppose $\sigma$ is a non-degenerated triangle with three vertices $\{\boldsymbol{q}_1, \boldsymbol{q}_2, \boldsymbol{q}_3\}$ ordered counterclockwise.  Let $f/\boldsymbol{F}$ be a scalar-/vector-valued function in $\mathcal{P}^1 (\sigma)/[\mathcal{P}^1 (\sigma)]^3$, respectively, $\{\boldsymbol{n},\boldsymbol{\tau}_1, \boldsymbol{\tau}_2\}$ forms an orthonormal basis. Then the discretized surface gradient operator $\nabla_S$ in \eqref{dis surface gradient} satisfies
\begin{subequations}
\label{eqn:alter gradient}
\begin{align}
\label{eqn:alter gradient scalar}
&\nabla_S f=(\partial_{\boldsymbol{\tau}_1}f)\,\boldsymbol{\tau}_1+(\partial_{\boldsymbol{\tau}_2}f)\,\boldsymbol{\tau}_2,\\[0.5em]
\label{eqn:alter gradient vector}
&\nabla_S \boldsymbol{F}=(\partial_{\boldsymbol{\tau}_1}\boldsymbol{F})\,\boldsymbol{\tau}_1^T+(\partial_{\boldsymbol{\tau}_2}\boldsymbol{F})\,\boldsymbol{\tau}_2^T,
\end{align}
\end{subequations}
where $\partial_{\boldsymbol{\tau}}f$ denotes the directional derivative of $f$ with respect to $\boldsymbol{\tau}$.
\end{lemma}
\begin{proof}
It suffices to prove \eqref{eqn:alter gradient scalar}. Let $\boldsymbol{x}=\lambda_1 \boldsymbol{q}_1+\lambda_2\boldsymbol{q}_2+\lambda_3\boldsymbol{q}_3$ with $\lambda_1+\lambda_2+\lambda_3=1$ be a point in $\sigma$. We observe that
\begin{align}
	 (\boldsymbol{q}_3-\boldsymbol{q}_2)\times \boldsymbol{n}\cdot (\boldsymbol{x}-\boldsymbol{q}_3)&= (\boldsymbol{x}-\boldsymbol{q}_3)\times (\boldsymbol{q}_3-\boldsymbol{q}_2)\cdot\boldsymbol{n}\nonumber\\
	&= (-\lambda_1 (\boldsymbol{q}_3-\boldsymbol{q}_1)-\lambda_2 (\boldsymbol{q}_3-\boldsymbol{q}_2))\times (\boldsymbol{q}_3-\boldsymbol{q}_2)\cdot\boldsymbol{n}\nonumber\\
	&=-\lambda_1 (\boldsymbol{q}_2-\boldsymbol{q}_1+\boldsymbol{q}_3-\boldsymbol{q}_2)\times (\boldsymbol{q}_3-\boldsymbol{q}_2)\cdot\boldsymbol{n}\nonumber\\
	&=-\lambda_1|\mathcal{J}\{\sigma\}|.
\end{align}
Thus $\lambda_1=\frac{ (\boldsymbol{q}_2-\boldsymbol{q}_3)\times \boldsymbol{n}}{|\mathcal{J}\{\sigma\}|}\cdot (\boldsymbol{x}-\boldsymbol{q}_3)$, and $\lambda_2, \lambda_3$ can be derived similarly.

By definition of the directional derivative, we deduce that.
\begin{align}
\partial_{\boldsymbol{\tau}_1}f&=\lim_{h\to 0}\frac{f (\boldsymbol{x}+h\boldsymbol{\tau}_1)-f (\boldsymbol{x})}{h}\nonumber\\
&=\lim_{h\to 0}\frac{1}{h}\left (f (\boldsymbol{q}_1)\frac{ (\boldsymbol{q}_2-\boldsymbol{q}_3)\times \boldsymbol{n}}{|\mathcal{J}\{\sigma\}|}\cdot (h\boldsymbol{\tau}_1)\right.\nonumber\\
&\quad \left.+f (\boldsymbol{q}_2)\frac{ (\boldsymbol{q}_3-\boldsymbol{q}_1)\times \boldsymbol{n}}{|\mathcal{J}\{\sigma\}|}\cdot (h\boldsymbol{\tau}_1)+f (\boldsymbol{q}_3)\frac{ (\boldsymbol{q}_1-\boldsymbol{q}_2)\times \boldsymbol{n}}{|\mathcal{J}\{\sigma\}|}\cdot (h\boldsymbol{\tau}_1)\right)\nonumber\\
&=\nabla_Sf\cdot \boldsymbol{\tau}_1.
\end{align}
Similarly, we have $\partial_{\boldsymbol{\tau}_2}f=\nabla_S f\cdot\boldsymbol{\tau}_2$.
Since $\{\boldsymbol{n}, \boldsymbol{\tau}_1, \boldsymbol{\tau}_2\}$ forms an orthonormal basis, by vector decomposition and $\nabla_S f\cdot\boldsymbol{n}=0$, we obtain
\begin{align}
	\nabla_S f&= (\nabla_S f\cdot\boldsymbol{n})\boldsymbol{n}+ (\nabla_S f\cdot\boldsymbol{\tau}_1)\boldsymbol{\tau}_1+ (\nabla_S f\cdot\boldsymbol{\tau}_2)\boldsymbol{\tau}_2\nonumber\\
	&=(\partial_{\boldsymbol{\tau}_1}f)\,\boldsymbol{\tau}_1+(\partial_{\boldsymbol{\tau}_2}f)\,\boldsymbol{\tau}_2,
\end{align}
which is the desired indentity.
\end{proof}

With the help of \eqref{eqn:alter gradient}, we can then give the following upperbound for the summand $\gamma (\boldsymbol{n})|\sigma|$ in the discretized energy $W$ \eqref{dis W}.
\begin{lemma}\label{lemma:gamma diff}
Suppose $\sigma, \bar{\sigma}$ are two non-degenerated triangles with counterclockwisely ordered vertices $\{\boldsymbol{q}_1, \boldsymbol{q}_2, \boldsymbol{q}_3\}, \{\bar{\boldsymbol{q}}_1, \bar{\boldsymbol{q}}_2, \bar{\boldsymbol{q}}_3\}$ and outward unit normal vectors $\boldsymbol{n}, \bar{\boldsymbol{n}}$, respectively. $\boldsymbol{X}$ is a vector valued function in $[\mathcal{P}^1 (\sigma)]^3$ satisfying $\boldsymbol{X} (\boldsymbol{q}_i)=\bar{\boldsymbol{q}}_i, i=1,2,3$. Then for any $k (\boldsymbol{n})\geq k_0 (\boldsymbol{n})$, the following inequality holds
\begin{equation}\label{aux ineq}
	\frac{1}{6}|\sigma|\sum_{i=1}^3 (\boldsymbol{Z}_k (\boldsymbol{n})\nabla_S \boldsymbol{X} ( (\boldsymbol{q}_i)^-)):\nabla_S\boldsymbol{X} ( (\boldsymbol{q}_i)^-)\geq \gamma (\bar{\boldsymbol{n}})|\bar{\sigma}|.
\end{equation}
\end{lemma}
\begin{proof}
Since $\boldsymbol{X}\in [\mathcal{P}^1 (\sigma)]^3$, its derivative $\nabla_S \boldsymbol{X}$ is a constant in $\sigma$. Suppose $\{\boldsymbol{n}, \boldsymbol{\tau}_1, \boldsymbol{\tau}_2\}$ forms an orthonormal basis, by applying \eqref{eqn:alter gradient vector}, we obtain
\begin{equation}
	\nabla_S\boldsymbol{X} ( (\boldsymbol{q}_i)^-)=(\partial_{\boldsymbol{\tau}_1}\boldsymbol{X})\,\boldsymbol{\tau}_1^T+(\partial_{\boldsymbol{\tau}_2}\boldsymbol{X})\,\boldsymbol{\tau}_2^T,\qquad i=1,2,3.
\end{equation}
Let $\partial_{\boldsymbol{\tau}_1}\boldsymbol{X}=s\boldsymbol{u}, \partial_{\boldsymbol{\tau}_2}\boldsymbol{X}=t\boldsymbol{v}$, where $s, t>0$ and $\boldsymbol{u}, \boldsymbol{v}\in \mathbb{S}^2$. Substituding this and the definition of $Z_k (\boldsymbol{n})$ \eqref{def of Z_k} into the LHS of \eqref{aux ineq} yields that
\begin{align}\label{lhs est}
&\frac{1}{6}|\sigma|\sum_{i=1}^3 (\boldsymbol{Z}_k (\boldsymbol{n})\nabla_S \boldsymbol{X} ( (\boldsymbol{q}_i)^-)):\nabla_S\boldsymbol{X} ( (\boldsymbol{q}_i)^-)\nonumber\\
&=\frac{1}{2}|\sigma|\left (\boldsymbol{Z}_k (\boldsymbol{n}) (s\boldsymbol{u}\boldsymbol{\tau}_1^T+t\boldsymbol{v}\boldsymbol{\tau}_2^T)\right): (s\boldsymbol{u}\boldsymbol{\tau}_1^T+t\boldsymbol{v}\boldsymbol{\tau}_2^T)\nonumber\\
&=\frac{1}{2}|\sigma|\left (s^2 (\boldsymbol{\tau}_1\cdot\boldsymbol{\tau}_1)\boldsymbol{u}^T\boldsymbol{Z}_k (\boldsymbol{n})\boldsymbol{u}+t^2 (\boldsymbol{\tau}_2\cdot\boldsymbol{\tau}_2)\boldsymbol{v}^T\boldsymbol{Z}_k (\boldsymbol{n})\boldsymbol{v}\right)\nonumber\\
&\geq |\sigma||st|\sqrt{F_k (\boldsymbol{n}, \boldsymbol{u}, \boldsymbol{v})}\geq |\sigma||st|\gamma (\boldsymbol{u}\times\boldsymbol{v}).
\end{align}
For the RHS of \eqref{aux ineq}, since $\bar{\sigma}=\boldsymbol{X} (\sigma)$, it holds that
\begin{equation}\label{rhs est}
	\gamma (\bar{n})|\bar{\sigma}|=\gamma (\bar{n})\int_{\sigma}|\partial_{\boldsymbol{\tau}_1}\boldsymbol{X}\times \partial_{\boldsymbol{\tau}_2}\boldsymbol{X}|dA=|\sigma||st|\gamma (\bar{\boldsymbol{n}})|\boldsymbol{u}\times\boldsymbol{v}|.
\end{equation}

Finally, since $\boldsymbol{X}\in [\mathcal{P}^1 (\sigma)]^3$, for $\boldsymbol{p}$ and $\boldsymbol{p}+h\boldsymbol{\tau}_1$ in $\sigma$, we have $\boldsymbol{X} (\boldsymbol{p}+h\boldsymbol{\tau}_1)$ and $\boldsymbol{X} (\boldsymbol{p})$ in $\bar{\sigma}$. From the definition of directional derivative for function in $[\mathcal{P}^1 (\sigma)]^3$, we get
\begin{equation}
 	s\boldsymbol{u}\cdot\bar{\boldsymbol{n}}=(\partial_{\boldsymbol{\tau}_1}\boldsymbol{X})\,\cdot\bar{\boldsymbol{n}}=\frac{\boldsymbol{X} (\boldsymbol{p}+h\boldsymbol{\tau}_1)-\boldsymbol{X} (\boldsymbol{p})}{h}\cdot\bar{\boldsymbol{n}}=0,
 \end{equation}
 and similarly $\boldsymbol{v}\cdot\bar{\boldsymbol{n}}=0$, thus $\gamma (\boldsymbol{u}\times\boldsymbol{v})=|\boldsymbol{u}\times\boldsymbol{v}|\gamma (\bar{\boldsymbol{n}})$. This equation together with \eqref{lhs est} and \eqref{rhs est} yield the desired inequality \eqref{aux ineq}.
\end{proof}

With the help of lemma \eqref{lemma:gamma diff}, we can then prove the energy stability part \eqref{dis W prop} in our main theorem \ref{thm: main}.
\begin{proof}
First for any $\boldsymbol{p}\in \mathbb{S}^2$, since $k (\boldsymbol{n})\geq k_0 (\boldsymbol{n})$, we have
\begin{equation}
	\boldsymbol{p}^T\boldsymbol{Z}_k (\boldsymbol{n})\boldsymbol{p}=\gamma (\boldsymbol{n})-2 (\boldsymbol{\xi}\cdot\boldsymbol{p}) (\boldsymbol{n}\cdot\boldsymbol{p})+k (\boldsymbol{n}) (\boldsymbol{n}\cdot\boldsymbol{p})^2\geq 0,
\end{equation}
thus $\boldsymbol{Z}_k (\boldsymbol{n})$ is positive definite. By Cauchy inequality, it holds that
\begin{align}\label{cauchy ineq}
	&\langle \boldsymbol{Z}_k (\boldsymbol{n}^m)\nabla_S\boldsymbol{X}^{m+1}, \nabla_S (\boldsymbol{X}^{m+1}-\boldsymbol{X}^m)\rangle_{S^m}\nonumber\\
	&\geq\frac{1}{2}\langle \boldsymbol{Z}_k (\boldsymbol{n}^m)\nabla_S\boldsymbol{X}^{m+1}, \nabla_S\boldsymbol{X}^{m+1}\rangle_{S^m}-\frac{1}{2}\langle \boldsymbol{Z}_k (\boldsymbol{n}^m)\nabla_S\boldsymbol{X}^{m}, \nabla_S\boldsymbol{X}^{m}\rangle_{S^m}.
\end{align}

Suppose $\{\boldsymbol{n}_j^m, \boldsymbol{\tau}_{j,1}^m, \boldsymbol{\tau}_{j,2}^m\}$ forms an orthonomal basis for $1\leq j\leq J$, by \eqref{eqn:alter gradient vector} we obtain
\begin{align}\label{W as inner product}
	&\frac{1}{2}\langle \boldsymbol{Z}_k (\boldsymbol{n}^m)\nabla_S\boldsymbol{X}^{m}, \nabla_S\boldsymbol{X}^{m}\rangle_{S^m}\nonumber\\
	&=\frac{1}{6}\sum_{j=1}^J\sum_{i=1}^3|\sigma_j^m| (\boldsymbol{Z}_k (\boldsymbol{n}_j^m)\nabla_S\boldsymbol{X}^m|_{\sigma_j^m} ( (\boldsymbol{q}_{j_i}^m)^-)):\nabla_S\boldsymbol{X}^m|_{\sigma_j^m} ( (\boldsymbol{q}_{j_i}^m)^-)\nonumber\\
	&=\frac{1}{2}\sum_{j=1}^J|\sigma_j^m| (\boldsymbol{\tau}_{j,1}^m\cdot \boldsymbol{Z}_k (\boldsymbol{n}_j^m)\boldsymbol{\tau}_{j,1}^m+\boldsymbol{\tau}_{j,2}^m\cdot Z_k (\boldsymbol{n}_j^m)\boldsymbol{\tau}_{j,2}^m)\nonumber\\
	&=\frac{1}{2}\sum_{j=1}^J|\sigma_j^m|\gamma (\boldsymbol{n}_j^m) (\boldsymbol{\tau}_{j,1}^m\cdot \boldsymbol{\tau}_{j,1}^m+\boldsymbol{\tau}_{j,2}^m\cdot \boldsymbol{\tau}_{j,2}^m)\nonumber\\
	&=\sum_{j=1}^J|\sigma_j^m|\gamma (\boldsymbol{n}_j^m)=W^m.
\end{align}
Then apply lemma \eqref{lemma:gamma diff} for $\sigma=\sigma_j^m, \bar{\sigma}=\sigma_j^{m+1}$ and $\boldsymbol{X}=\boldsymbol{X}^{m+1}|_{\sigma_j^m}$, we know that
\begin{equation}\label{apply lem}
	\frac{1}{6}|\sigma_j^m|\sum_{i=1}^3 (\boldsymbol{Z}_k (\boldsymbol{n}_j^m)\nabla_S\boldsymbol{X}|_{\sigma_j^m} ( (\boldsymbol{q}_i)^-)):\nabla_S\boldsymbol{X}|_{\sigma_j^m} ( (\boldsymbol{q}_i)^-)\geq \gamma (\boldsymbol{n}_j^{m+1})|\sigma_j^{m+1}|.
\end{equation}
And this inequality holds for all $1\leq j\leq J$. Summing \eqref{apply lem} for $j=1, 2,\ldots, J$ and combining \eqref{cauchy ineq} and \eqref{W as inner product} yields that
\begin{equation}\label{W difference between two step}
	\langle \boldsymbol{Z}_k (\boldsymbol{n}^m)\nabla_S\boldsymbol{X}^{m+1}, \nabla_S (\boldsymbol{X}^{m+1}-\boldsymbol{X}^m)\rangle_{S^m}\geq W^{m+1}-W^m.
\end{equation}

Finally, choosing $\psi=\mu^{m+1}$ in \eqref{eqn:SP-PFEM eq1} and $\boldsymbol{\omega}=\boldsymbol{X}^{m+1}$ in \eqref{eqn:SP-PFEM eq2}, together with \eqref{W difference between two step} yields that
\begin{equation}
	W^{m+1}-W^m\leq \tau\left (\nabla_S\mu^{m+1}, \nabla_S\mu^{m+1}\right)_{S^m}\leq 0.
\end{equation}
Since this inequality is valid for all $m$, the unconditionally energy stable part \eqref{dis W prop} in theorem \ref{thm: main} is proved.
\end{proof}

\section{Numerical results}
In this section, we first state the setup for solving the SP-PFEM \eqref{eqn:SP-PFEM}. Then we present serval numerical computations, including the convergence test and the structural preserving test. Finally, we apply \eqref{eqn:SP-PFEM} to simulate the surface evolution for different anisotropic energies.

The minimal stabilizing function $k_0 (\boldsymbol{n})$ is given by the bilinear interpolation, where the interpolation points are $\boldsymbol{n}_{i,j}= (\cos\phi_i\cos\theta_j, \cos\phi_i\sin\theta_j,\sin\phi_i)^T, \phi_i=-\frac{\pi}{2}+\frac{i}{10}\pi, \theta_j=-\pi+\frac{j}{5}\pi,\, 0\leq i,j\leq 10$, and the $k_0 (\boldsymbol{n}_{i,j})$ is given by solving \eqref{def of k_0}. The surface energy matrix $Z_k (\boldsymbol{n})$ as well as the SP-PFEM \eqref{eqn:SP-PFEM} is thus determined by giving a stabilizing function $k (\boldsymbol{n})\geq k_0 (\boldsymbol{n})$.

The fully-implicit linear system \eqref{eqn:SP-PFEM} is solved by using the Newton's iterative method provided in \cite{bao2021structurepreserving}. And for each discrete time level $t_m=m\tau$, the iteration is terminated when $\left\|\boldsymbol{X}^{\delta}\right\|_\infty\leq 10^{-12}, \left\|\mu^{\delta}\right\|\leq 10^{-12}$, where $ (\boldsymbol{X}^{\delta}(\cdot), \mu^{\delta}(\cdot))\in [\mathbb{K}^m]^3\times \mathbb{K}^m$ is the Newton direction.

Given a initial shape $S_0$, we generate its approximation $S^0=\cup_{j=1}^J\bar{\sigma}_j^0$ with $J$ triangles $\{\sigma_j^0\}_{j=1}^J$ and $I$ vertices $\{\boldsymbol{q}_i^0\}_{i=1}^I$ by using a matlab toolbox called \textit{CFDTool} \cite{CFDTool} with a given parameter mesh size $h$. For the time step size $\tau$ and the mesh size $h$, we denote the solution of \eqref{eqn:SP-PFEM} with the initial approximation $S^0_h$ with $J (h)$ triangles and $I (h)$ vertices at $t_m$ by $ (\boldsymbol{X}_{h, \tau}^m, \mu_{h,\tau}^m)$. And we define $\boldsymbol{X}_{h, \tau} (t)$ by
\begin{equation}
	\boldsymbol{X}_{h,\tau} (t)=\frac{t-t_m}{\tau}\boldsymbol{X}_{h,\tau}^m+\frac{t_{m+1}-t}{\tau}\boldsymbol{X}_{h,\tau}^{m+1},\qquad \forall t\in [t_m, t_{m+1}),\, m\geq 0.
\end{equation}
And $S_{h,\tau} (t)$ is defined similarly.

To test the convergence rate of \eqref{eqn:SP-PFEM}, we adopt the manifold distance $M (\cdot, \cdot)$ to measure the difference between two closed surfaces $S_1$ and $S_2$, which is given by
\begin{equation}
	M (S_1,S_2):=2|\Omega_1\cup\Omega_2|-|\Omega_1|-|\Omega_2|,
\end{equation}
where we denote $\Omega_1$ and $\Omega_2$ to be the interior of $S_1$ and $S_2$, respectively. Based on the manifold distance, the numerical error is defined as
\begin{equation}
	e_{h,\tau} (t):=M (S_{h,\tau} (t), S (t)).
\end{equation}
Here $S (t)$ is approximated by the refined mesh $S_{h_e,\tau_e} (t)$ with $k (\boldsymbol{n})=k_0 (\boldsymbol{n})$, where $h_e=2^{-4}$ and $\tau_e=\frac{2}{25}h_e^2$.

In the numerical experiments for convergence rates, the time step size and the mesh size are chosen as $\tau=\frac{2}{25}h^2$, the initial shape $S_0$ is chosen as a $2\times 2\times 1$ cuboid, and its approximation is a polyhedron $S_{h_e, \tau_e}^0$ with $10718$ triangles and $5361$ vertices. We consider the following five cases of the anisotropic surface energy $\gamma (\boldsymbol{n})$ as well as the stabilizing function $k (\boldsymbol{n})$:
\begin{itemize}
\item Case 1: $\gamma (\boldsymbol{n})=1+\frac{1}{4} (n_1^4+n_2^4+n_3^4)$, $k (\boldsymbol{n})=k_0 (\boldsymbol{n})$;
\item Case 2: $\gamma (\boldsymbol{n})=1+\frac{1}{2} (n_1^4+n_2^4+n_3^4)$, $k (\boldsymbol{n})=k_0 (\boldsymbol{n})$;
\item Case 3: $\gamma (\boldsymbol{n})= (n_1^4+n_2^4+n_3^4)^{\frac{1}{4}}$, $k (\boldsymbol{n})=k_0 (\boldsymbol{n})$;
\item Case 4: $\gamma (\boldsymbol{n})= (n_1^4+n_2^4+n_3^4)^{\frac{1}{4}}$, $k (\boldsymbol{n})=k_0 (\boldsymbol{n})+1$;
\item Case 5: $\gamma (\boldsymbol{n})= (n_1^4+n_2^4+n_3^4)^{\frac{1}{4}}$, $k (\boldsymbol{n})=k_0 (\boldsymbol{n})+2$;
\item Case 6: $\gamma (\boldsymbol{n})= (n_1^4+n_2^4+n_3^4)^{\frac{1}{4}}$, $k (\boldsymbol{n})=k_0 (\boldsymbol{n})+5$.
\end{itemize}
The numerical errors are listed in TABLE \ref{tb: convergence rate}. We note that while the $\gamma (\boldsymbol{n})$ and $k (\boldsymbol{n})$ are different in each cases, the convergence rates for this manifold error are all about second order in $h$. This result indicates the proposed SP-PFEM \eqref{eqn:SP-PFEM} has a good robustness in convergence rate, and we can choose large $k (\boldsymbol{n})$ such as $k (\boldsymbol{n})\equiv \sup\limits_{\boldsymbol{n}\in\mathbb{S}^2} k_0 (\boldsymbol{n})$ to avoid the computation cost in bilinear interpolation without loss of efficiency.
\begin{table}[tbh]
\label{tb: convergence rate}
\begin{center}
\begin{tabular}{@{\extracolsep{\fill}}|c|cc|cc|cc|}\hline
$ (h,\tau)$ &$e_{h,\tau} (\frac{1}{2})$ \footnotesize{Case 1} & \footnotesize{order} &$e_{h,\tau} (\frac{1}{2})$ \footnotesize{Case 2} & \footnotesize{order} &$e_{h,\tau} (\frac{1}{2})$ \footnotesize{Case 3} & \footnotesize{order} \\ \hline
$ (h_0,\tau_0)$ & 1.24E-1 & - &1.47E-1 &-& 1.12E-1 &- \\ \hline
$ (\frac{h_0}{2},~\frac{\tau_0}{4})$ & 3.06E-2 & 2.01 &3.54E-2 &2.05& 2.82E-2 &1.98
\\ \hline
$ (\frac{h_0}{2^2},~\frac{\tau_0}{4^2})$ & 7.90E-3 & 1.96 &8.74E-3 &2.02& 7.54E-3 &1.90 \\ \hline

 \end{tabular}
 \begin{tabular}{@{\extracolsep{\fill}}|c|cc|cc|cc|}\hline
$ (h,\tau)$ &$e_{h,\tau} (\frac{1}{2})$ \footnotesize{Case 4} & \footnotesize{order} &$e_{h,\tau} (\frac{1}{2})$ \footnotesize{Case 5} & \footnotesize{order} &$e_{h,\tau} (\frac{1}{2})$ \footnotesize{Case 6} & \footnotesize{order} \\ \hline
$ (h_0,\tau_0)$ & 1.10E-1 & - &1.12E-1 &-& 1.12E-1 &- \\ \hline
$ (\frac{h_0}{2},~\frac{\tau_0}{4})$ & 2.83E-2 & 1.96 &2.89E-2 &1.96& 3.09E-2 &1.99
\\ \hline
$ (\frac{h_0}{2^2},~\frac{\tau_0}{4^2})$ & 7.48E-3 & 1.92 &7.58E-3 &1.93& 7.86E-3 &1.97 \\ \hline

 \end{tabular}\vspace{0.8em}
 \begin{tabular}{@{\extracolsep{\fill}}|c|cc|cc|cc|}\hline
$ (h,\tau)$ &$e_{h,\tau} (1)$\, \footnotesize{Case 1} & \footnotesize{order} &$e_{h,\tau} (1)$\, \footnotesize{Case 2} & \footnotesize{order} &$e_{h,\tau} (1)$ \,\footnotesize{Case 3} & \footnotesize{order} \\ \hline
$ (h_0,\tau_0)$ &1.46E-1 &-& 1.22E-1 & - & 1.11E-1 &- \\ \hline
$ (\frac{h_0}{2},~\frac{\tau_0}{4})$ &3.52E-2 &2.05& 3.01E-2 & 2.02 & 2.74E-2 &2.02
\\ \hline
$ (\frac{h_0}{2^2},~\frac{\tau_0}{4^2})$ &8.67E-3 &2.02& 7.75E-3 & 1.96 & 7.21E-3 &1.93 \\ \hline
 \end{tabular}
\begin{tabular}{@{\extracolsep{\fill}}|c|cc|cc|cc|}\hline
$ (h,\tau)$ &$e_{h,\tau} (1)$ \,\footnotesize{Case 4} & \footnotesize{order} &$e_{h,\tau} (1)$\, \footnotesize{Case 5} & \footnotesize{order} &$e_{h,\tau} (1)$\, \footnotesize{Case 6} & \footnotesize{order} \\ \hline
$ (h_0,\tau_0)$ & 1.10E-1 & - &1.10E-1 &-& 1.13E-1 &- \\ \hline
$ (\frac{h_0}{2},~\frac{\tau_0}{4})$ & 2.76E-2 & 1.99 &2.80E-2 &1.97& 2.90E-2 &1.96
\\ \hline
$ (\frac{h_0}{2^2},~\frac{\tau_0}{4^2})$ & 7.23E-3 & 1.93 &7.36E-3 &1.93& 7.56E-3 &1.94 \\ \hline

 \end{tabular}\vspace{1em}
\end{center}
\caption{Numerical error $e_{h,\tau}$ at time $T=\frac{1}{2}, 1$ and the convergence rate for simulating the anisotropic surface diffusion start from a $2\times 2\times 1$ cuboid with different anisotropic energies $\gamma (\boldsymbol{n})$ and stabilizing functions $k (\boldsymbol{n})$ given in Case 1-6. The mesh size, time step size, number of triangles and number of vertices for the coarse shapes are $ (h_0:=2^{-1}, \tau_0:=\frac{2^{-1}}{25}, 140, 72)$, and then $ (2^{-2}, \frac{2^{-3}}{25}, 624, 314),\, (2^{-3}, \frac{2^{-5}}{25}, 2502, 1253)$. }
 \end{table}

To examine the volume conservation and unconditionally energy dissipation, we consider these two indicators, the normalized volume change $\frac{\Delta V (t)}{V (0)}:=\frac{V (t)-V (0)}{V (0)}$ and the normalized energy $\frac{W (t)}{W (0)}$; and we choose the initial shape to be a $2\times2\times1$ ellipsoid. Figure \ref{fig: volume} shows the normalized volume change $\frac{\Delta V (t)}{V (0)}$ for the anisotropy in Case 1, Case 2, Case 3, with fixed $h=2^{-3}, \tau=\frac{2}{25}h^2$ in (a), (b), (c), respectively. We find the order of magnitude of the volume change $\Delta V (t)$ is $10^{-15}$, which is close to the machine epsilon $10^{-16}$, and thus indicates the volume is well conserved. Figure \ref{fig: energy1} plot the normalized energy $\frac{W (t)}{W (0)}$ for different cases and mesh size $h$ with $\tau=\frac{2}{25}h^2$ and for different $\tau$ with a constant mesh size $h=2^{-4}$, respectively. We observe the normalized energy $\frac{W (t)}{W (0)}$ is monotonically decreasing in time, even for the relatively large time step size $\tau=0.01$. And these graphs also suggest the stabilizing function $k(\boldsymbol{n})$ does not infect the energy, and we can choose a relatively large stabilizing function $k(\boldsymbol{n})$, which consists the result in convergence rate test. And the above volume and energy tests validate the theorem \eqref{thm: main} numerically.

\begin{figure}[hbtp!]
\centering
\includegraphics[width=1\textwidth]{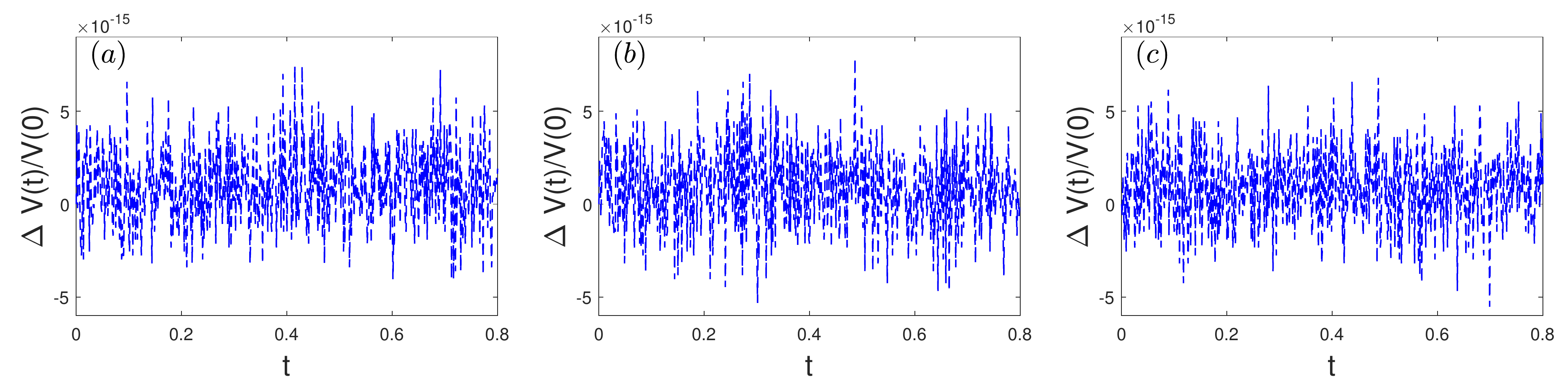}
\caption{Plot of the normalized volume change $\frac{\Delta V(t)}{V(0)}$ for different anisotropic energies in Case 1 (a), Case 2 (b) and Case 3 (c).}
\label{fig: volume}
\end{figure}
\begin{figure}[hbtp!]
\centering
\includegraphics[width=1\textwidth]{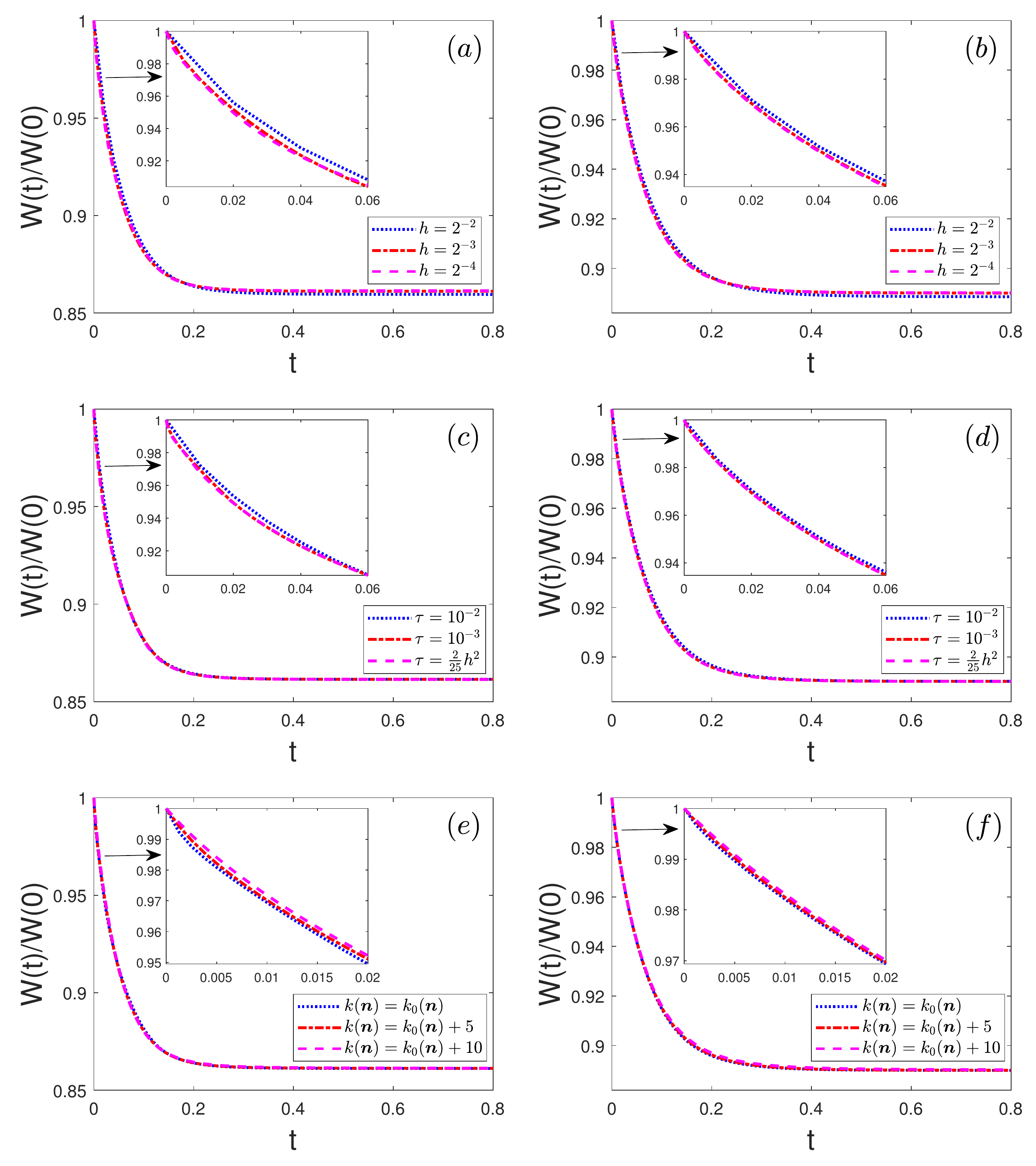}
\caption{Plot of the normalized energy $\frac{W(t)}{W(0)}$ for weak/strong anisotropy $\gamma(\boldsymbol{n})=1+\frac{1}{4}(n_1^4+n_2^4+n_3^4)$ or $\gamma(\boldsymbol{n})=1+\frac{1}{2}(n_1^4+n_2^4+n_3^4)$ for $k(\boldsymbol{n})=k_0(\boldsymbol{n})$ with different $h$ and $\tau$ (a), (b); for fixed $h=2^{-4}$ with different $\tau$ (c), (d); for $h=2^{-4}, \tau=\frac{2}{25}h^2$ with different $k(\boldsymbol{n})$ (e), (f), respectively.}
\label{fig: energy1}
\end{figure}

Finally, we use \eqref{eqn:SP-PFEM} to investigate the motion by anisotropic surface diffusion with different anisotropies. We consider the weak anisotropy $\gamma (\boldsymbol{n})=\sqrt{n_1^2+n_2^2+2n_3^2}$ with $k (\boldsymbol{n})=k_0 (\boldsymbol{n})$ first. The evolutions of a smooth $2\times 2\times 1$ ellipsoid and a non-smooth $2\times 2\times 1$ cuboid are shown in figure \ref{fig: evolve1} and figure \ref{fig: evolve2}, respectively. We choose the mesh size $h=2^{-4}$ and the time step size $\tau=\frac{2}{25}h^2$, and the ellipsoid and the cuboid are initially approximated by $K (h) = 10718, I(h) = 5361$ and $K(h) = 32768, I (h) = 16386$, respectively. By comparing the two figures, we find the two numerical equilibriums are close in shape, which indicates our scheme \eqref{eqn:SP-PFEM} is stable in catching the equilibrium shape for different initial shapes. We can see that the meshes are well distributed during the evolution, and we do not need to remesh the surface.

\begin{figure}[hbtp!]
\centering
\includegraphics[width=1\textwidth]{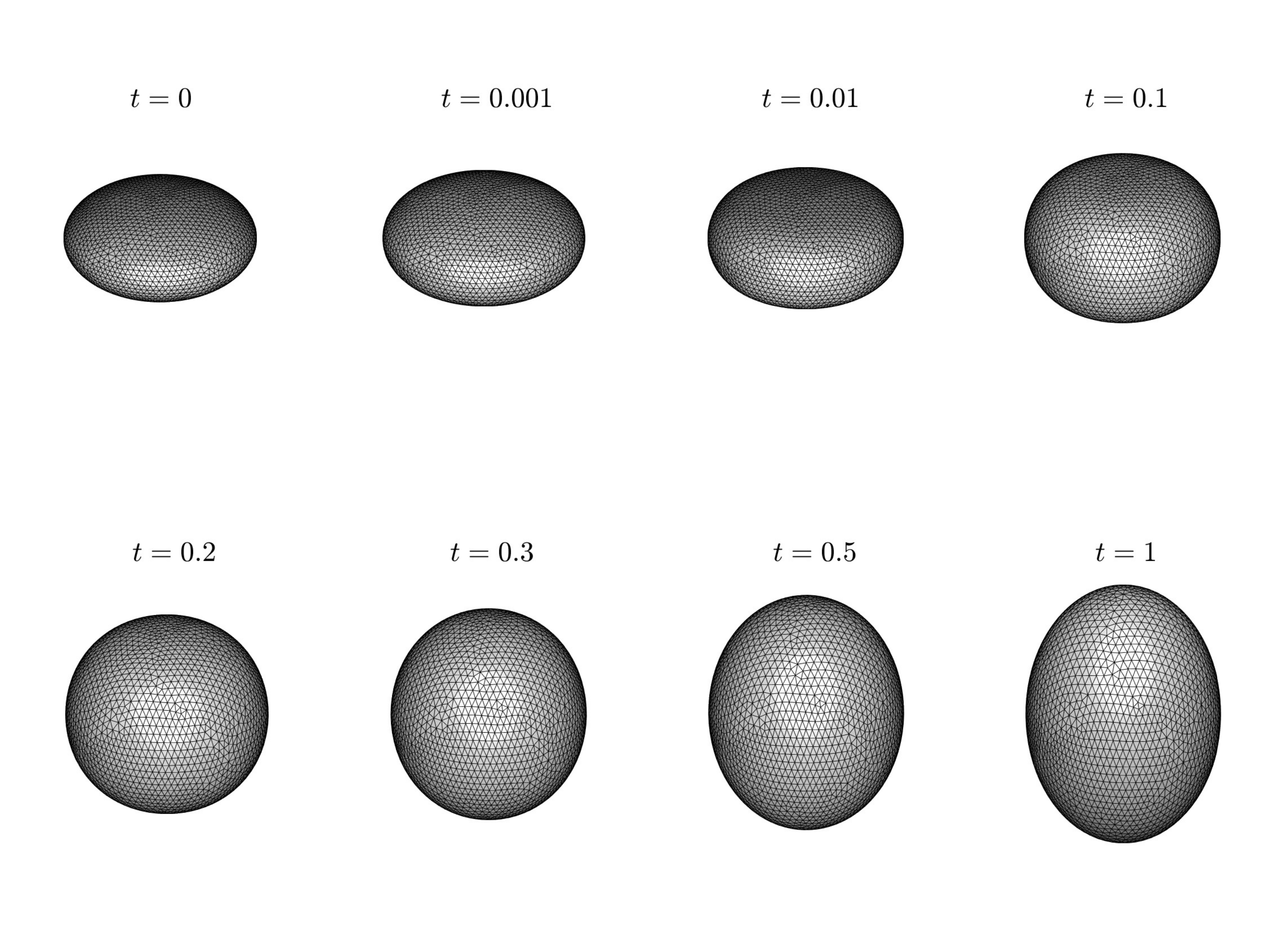}
\caption{Evolution of a $2\times 2\times 1$ ellipsoid by anisotropic surface diffusion with a weak anisotropy $\gamma(\boldsymbol{n})=\sqrt{n_1^2+n_2^2+2n_3^2}$ and $k(\boldsymbol{n})=k_0(\boldsymbol{n})$ at different times.}
\label{fig: evolve1}
\end{figure}

\begin{figure}[hbtp!]
\centering
\includegraphics[width=1\textwidth]{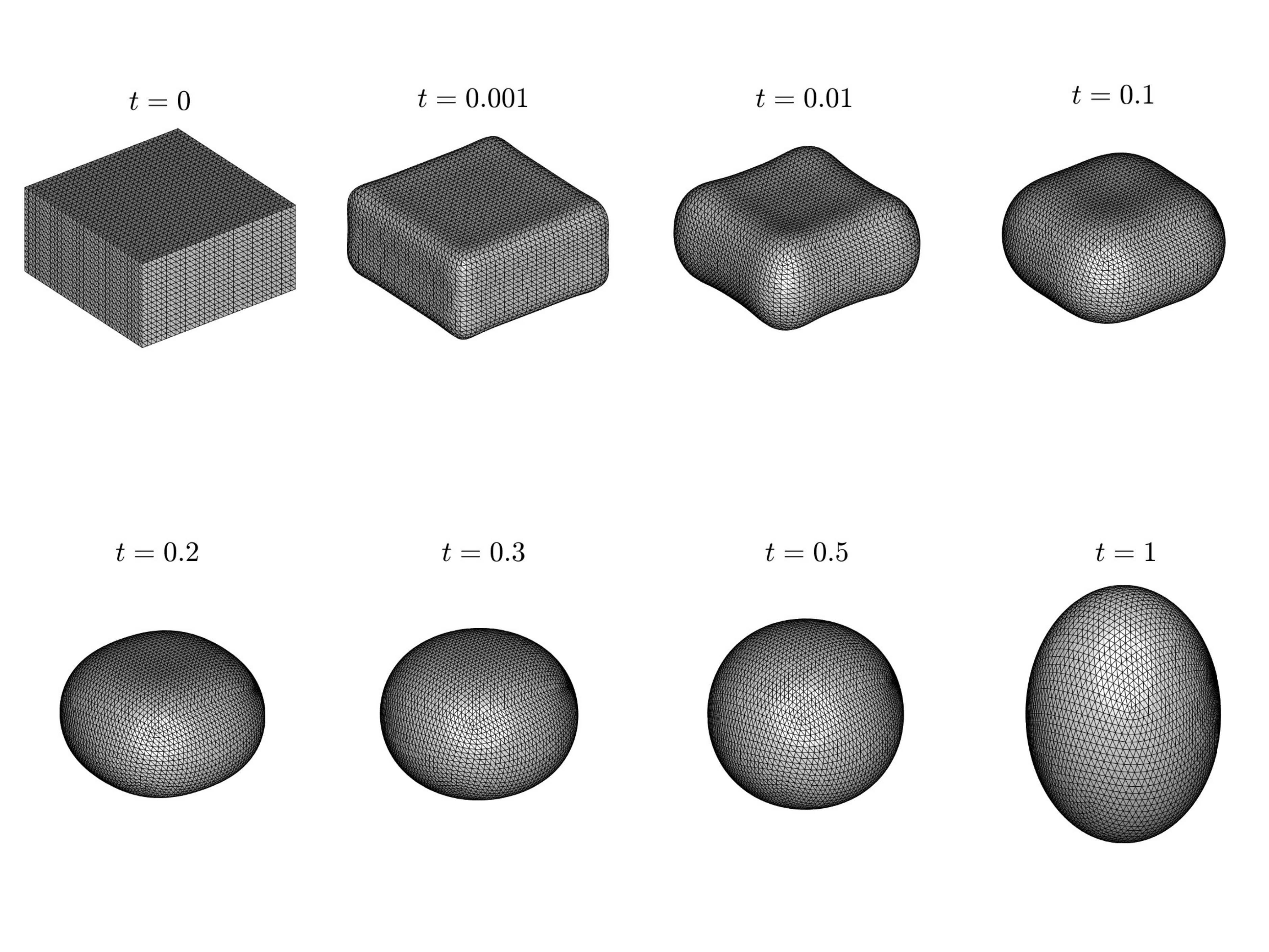}
\caption{Evolution of a $2\times 2\times 1$ cuboid by anisotropic surface diffusion with a weak anisotropy $\gamma(\boldsymbol{n})=\sqrt{n_1^2+n_2^2+2n_3^2}$ and $k(\boldsymbol{n})=k_0(\boldsymbol{n})$ at different times.}
\label{fig: evolve2}
\end{figure}

Then we show the evolution of a strong anisotropy $\gamma (\boldsymbol{n})=1+\frac{1}{2} (n_1^4+n_2^4+n_3^4)$ from a $2\times 2\times 1$ cuboid, and the parameters are chosen the same as in previous weak anisotropy. As can be seen from figure \ref{fig: evolve3}, the large and flat facets may be broken into small facets, and the small facets may also merge into a large facet. Moreover, we note from figure \ref{fig: evolve3} that the triangulations become dense at the edges where the facets merge but become sparse at the other edges and at the interior of the facets where the weighted mean curvature $\mu$ is almost a constant.

\begin{figure}[hbtp!]
\centering
\includegraphics[width=1\textwidth]{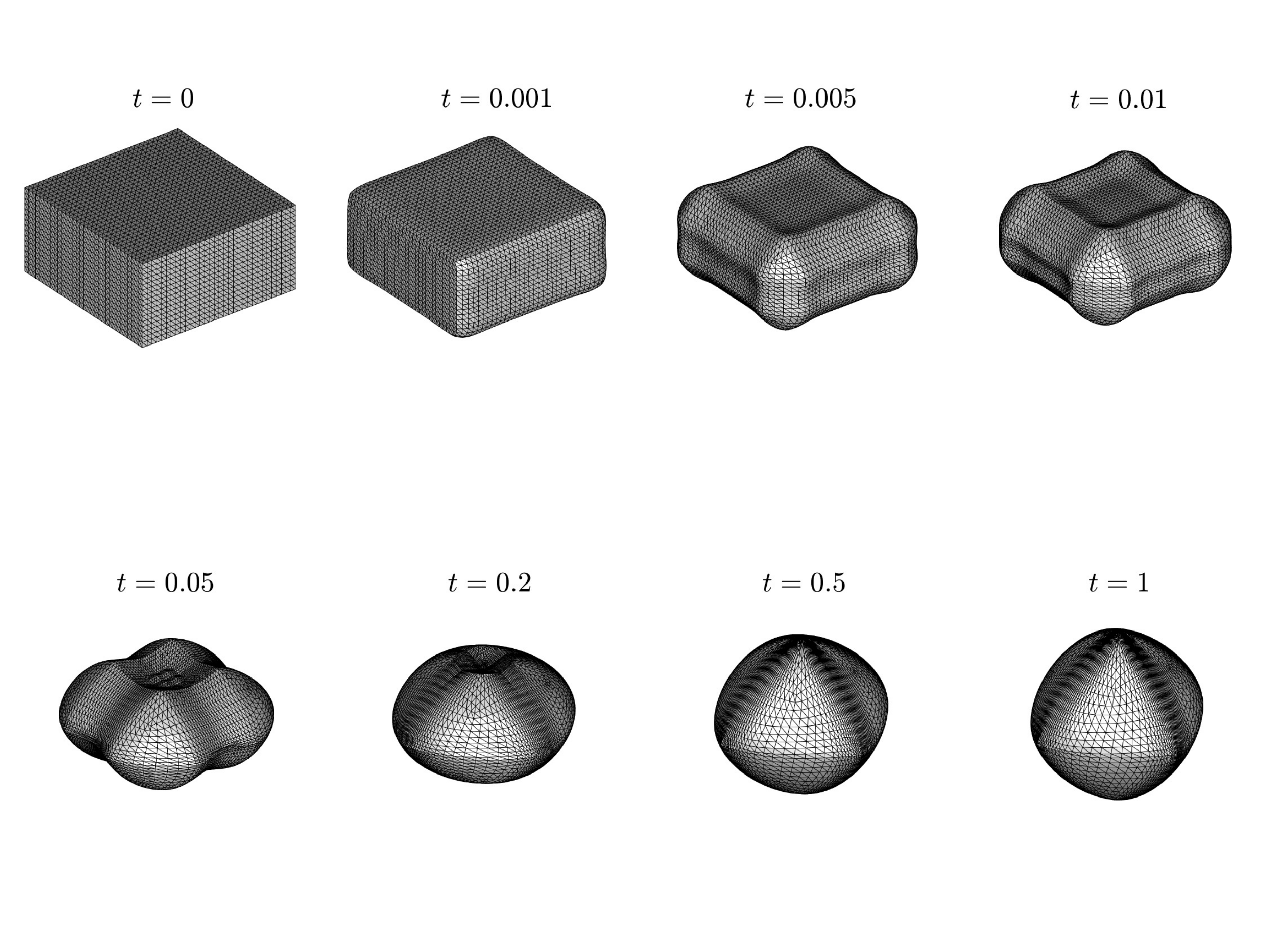}
\caption{Evolution of a $2\times 2\times 1$ cuboid by anisotropic surface diffusion with a strong anisotropy $\gamma(\boldsymbol{n})=1+\frac{1}{2}(n_1^4+n_2^4+n_3^4)$ and $k(\boldsymbol{n})=k_0(\boldsymbol{n})$ at different times.}
\label{fig: evolve3}
\end{figure}

\section{Conclusions}

By generalizing the symmetrized surface energy matrix $Z_k(\boldsymbol{n})$ into 3D, we proposed a new weak formulation for the weighted mean curvature $\mu$ and derived a symmetrized variational formulation. Based on this new variational formulation, we proposed a structural-preserving finite element method (SP-PFEM) for anisotropic surface diffusion in 3D and established its unconditional energy stability for $C^2$ anisotropies with $\gamma(-\boldsymbol{n})=\gamma(\boldsymbol{n})$ \eqref{engstabgmp}. Moreover, we constructed the upper bound for the minimal stabilizing function $k_0(\boldsymbol{n})$, which also gave a promising approach to determine the symmetrized surface energy matrix $\boldsymbol{Z}_k(\boldsymbol{n})$. Unlike other structural-preserving schemes for anisotropic surface diffusion, our SP-PFEM can work for an arbitrary initial shape and a much broader functional class, the symmetric $C^2$ functions.

Similar to other PFEMs, our SP-PFEM for 3D anisotropic surface diffusion illustrated a second-order convergence rate, which was also verified by various numerical experiments. We examined the volume is conserved in machine epsilon, and the energy is dissipative, regardless of the choice of the anisotropy $\gamma(\boldsymbol{n})$, the stabilizing function $k(\boldsymbol{n})$ and the time step size $\tau$, which matched the prediction of the main theorem \ref{thm: main} well. We presented the evolution of both smooth and non-smooth initial shapes with weakly/strongly anisotropic energies. Many interesting phenomena were shown, such as numerical equilibrium, facet breaking, and facet merging.

Finally, We point out that the symmetrized variational formulation \eqref{eqn:New 3Daniso} can be applied to other geometric flows with symmetric anisotropic surface energy, such as the anisotropic mean curvature flow \cite{barrett2008numerical}, the Stefan problem \cite{barrett2010stable}, and the anisotropic elastic flow \cite{barrett2012parametric}. Our future work will consider the SP-PFEM for asymmetric anisotropic energies for the 2D and 3D anisotropic surface diffusion.



\bigskip

\begin{center}
{\textbf{Appendix A.}} Remark for several common used anisotropic surface energies
\end{center}
\setcounter{equation}{0}
\renewcommand{\theequation}{A.\arabic{equation}}

\medskip
For the ellipsoidal anisotropic surface energy \cite{barrett2008numerical}
\begin{equation}
	\gamma(\boldsymbol{n})=\sqrt{\boldsymbol{n}^T\boldsymbol{G}\boldsymbol{n}},
\end{equation}where $\boldsymbol{G}$ is positive definte,  we have
\begin{align}\label{ellipsoidalaa1}
&\gamma(\boldsymbol{p})=\sqrt{\boldsymbol{p}^TG\boldsymbol{p}},
\qquad \forall\boldsymbol{p}\in \mathbb{R}^3_*:=\mathbb{R}^3\setminus \{\boldsymbol{0}\},\\
&\boldsymbol{\xi}=\boldsymbol{\xi}(\boldsymbol{n})=
 \gamma(\boldsymbol{n})^{-1}\,\boldsymbol{G}\,\boldsymbol{n}, \quad  \forall \boldsymbol{n}\in \mathbb{S}^1, \\
 &
 {\bf H}_\gamma(\boldsymbol{n})=\gamma(\boldsymbol{n})^{-3/2}(\gamma(\boldsymbol{n})^2\boldsymbol{G}-(\boldsymbol{G}\boldsymbol{n})(\boldsymbol{G}\boldsymbol{n})^T).
\end{align}
And we know ${\bf H}_\gamma(\boldsymbol{n})$ is semi-positive definite by Cauchy inequality, which indicates the ellipsoidal anisotropy is weakly anisotropic.

For the $l^r$-norm ($r\geq 2$) metric anisotropic surface energy \cite{bao2021symmetrized}
\begin{equation}
 	\gamma(\boldsymbol{n})=(|n_1|^r+|n_2|^r+|n_3|^r)^{1/r},
 \end{equation} we have
\begin{align}
&\gamma(\boldsymbol{p})=\left\|\boldsymbol{p}\right\|_{l^r}=
 \left(|p_1|^r+|p_2|^r+|p_3|^r\right)^{\frac{1}{r}}, \qquad \forall\boldsymbol{p}\in \mathbb{R}^3_*\,,\\
&\boldsymbol{\xi}=\boldsymbol{\xi}(\boldsymbol{n})=
 \gamma(\boldsymbol{n})^{1-r}
  \begin{pmatrix}|n_1|^{r-2}n_1\\ |n_2|^{r-2}n_2\\|n_3|^{r-2}n_3\end{pmatrix}, \qquad \forall\boldsymbol{n}\in \mathbb{S}^1,\\
  &{\bf H}_\gamma(\boldsymbol{n})=(r-1)\gamma(\boldsymbol{n})^{1-2r}\begin{pmatrix}|n_1|^{r-2}(|n_2|^r+|n_3|^r)&*&*\\ -|n_1n_2|^{r-2}n_1n_2&*&*\\-|n_1n_3|^{r-2}n_1n_3&*&*\end{pmatrix}.
\end{align}
Where the $*$ entries can be deduced from other entries. By checking leading principal minors, we know that ${\bf H}_{\gamma}(\boldsymbol{n})$ is semi-positive definite. Thus the $l^r$-norm anisotropy is weakly anisotropic.

For the $4$-fold anisotropic surface energy \cite{deckelnick2005computation}
\begin{equation}
	\gamma(\boldsymbol{n})=1+\beta(n_1^4+n_2^4+n_3^4),
\end{equation}we have
\begin{align}
&\gamma(\boldsymbol{p})=\left(p_1^2+p_2^2+p_3^2\right)^{\frac{1}{2}}+\beta (p_1^4+p_2^4+p_3^4)\left(p_1^2+p_2^2+p_3^2\right)^{-\frac{3}{2}}, \\
&\boldsymbol{\xi}=\boldsymbol{\xi}(\boldsymbol{n})=\boldsymbol{n}+\beta \left(4n_1^3-3n_1(n_1^4+n_2^4+n_3^4), *, *\right)^T,\\
&\lambda_1({\boldsymbol{n}})+\lambda_2(\boldsymbol{n})=2(1-3\beta)+36\beta(n_1^2n_2^2+n_2^2n_3^2+n_3^2n_1^2).
\end{align}
Thus $\gamma(\boldsymbol{n})$ is strongly anisotropic if $\beta>1/3$. For $\beta=1/3$, we know $\lambda_1(\boldsymbol{n})+\lambda_2(\boldsymbol{n})\geq 0$ and
\begin{equation}
	\lambda_1(\boldsymbol{n})\lambda_2(\boldsymbol{n})=4(5(n_1^4n_2^4+n_2^4n_3^4+n_3^4n_1^4)+18n_1^2n_2^2n_3^2))\geq 0,
\end{equation}
which means $\gamma(\boldsymbol{n})$ is weakly anisotropic. When $\beta=0$, $\gamma(\boldsymbol{n})$ collapse to the $l^2$-norm, and we have alreadly known such $\gamma(\boldsymbol{n})$ is weakly anisotropic. We know that $\gamma(\boldsymbol{n})$ is weakly anisotropic for $0\leq \beta\leq \frac{1}{3}$ and is strongly anisotropic for $\beta>\frac{1}{3}$.

Finally,  for the regularized BGN anisotropic surface energy \cite{barrett2008variational}
\begin{equation}
	 \gamma(\boldsymbol{n})=\left(\sum_{l=1}^L (\boldsymbol{n}^T\boldsymbol{G}_l\boldsymbol{n})^{r/2}\right)^{1/r},
\end{equation} where $\boldsymbol{G}_1, \boldsymbol{G}_2,\ldots, \boldsymbol{G}_L$ are positive definite matrices, we get
\begin{align}
&\gamma(\boldsymbol{p})=\left(\sum_{l=1}^L (\boldsymbol{p}^T\boldsymbol{G}_l\boldsymbol{p})^{r/2}\right)^{1/r},\qquad \forall\boldsymbol{p}\in \mathbb{R}^3_*,\\
&\boldsymbol{\xi}=\boldsymbol{\xi}({\boldsymbol{n}})=\gamma(\boldsymbol{n})^{1-r}\sum_{l=1}^L \gamma^{r-2}_l(\boldsymbol{n})\boldsymbol{G}_l\boldsymbol{n}\quad \forall\boldsymbol{n}\in \mathbb{S}^1,\\
&{\bf H}_\gamma(\boldsymbol{n})=\gamma(\boldsymbol{n})^{1-2r}(\boldsymbol{M}_1+(r-1)\boldsymbol{M}_2).
\end{align}
where $\gamma_l(\boldsymbol{n}):=\sqrt{\boldsymbol{n}^T\boldsymbol{G}_l\boldsymbol{n}}, \,l=1,2,\ldots,L$, and
\begin{equation}
	\boldsymbol{M}_1=\gamma(\boldsymbol{n})^r\sum_{l=1}^L\gamma_l(\boldsymbol{n})^{r-4}(\gamma_l(\boldsymbol{n})^2\boldsymbol{G}_l-(\boldsymbol{G}_l\boldsymbol{n})(\boldsymbol{G}_l\boldsymbol{n})^T),
\end{equation}
\begin{equation}
	\boldsymbol{M}_2=\gamma(\boldsymbol{n})^r\sum_{l=1}^L(\boldsymbol{G}_l\boldsymbol{n})(\boldsymbol{G}_l\boldsymbol{n})^T\gamma_l^{r-4}(\boldsymbol{n})-(\sum_{l=1}^L \gamma_l^{r-2}(\boldsymbol{n})\boldsymbol{G}_l\boldsymbol{n})(\sum_{l=1}^L \gamma_l^{r-2}(\boldsymbol{n})\boldsymbol{G}_l\boldsymbol{n})^T.
\end{equation}
By Cauchy inequality, we obtain that $\boldsymbol{M}_1, \boldsymbol{M}_2$ are semi-positive definite. Thus the BGN anisotropy is weakly anisotropic for $r\geq 1$.


\bibliographystyle{siamplain}
\bibliography{thebib}

\begin{thebibliography}{10}

\bibitem{armelao2006recent}
{\sc L.~Armelao, D.~Barreca, G.~Bottaro, A.~Gasparotto, S.~Gross, C.~Maragno,
  and E.~Tondello}, {\em Recent trends on nanocomposites based on cu, ag and au
  clusters: A closer look}, Coordination Chemistry Reviews, 250 (2006),
  pp.~1294--1314.

\bibitem{asaro1972interface}
{\sc R.~Asaro and W.~Tiller}, {\em Interface morphology development during
  stress corrosion cracking: Part i. via surface diffusion}, Metallurgical and
  Materials Transactions B, 3 (1972), pp.~1789--1796.

\bibitem{bansch2004surface}
{\sc E.~B{\"a}nsch, P.~Morin, and R.~H. Nochetto}, {\em Surface diffusion of
  graphs: variational formulation, error analysis, and simulation}, SIAM J.
  Numer. Anal., 42 (2004), pp.~773--799.

\bibitem{bao2022volume}
{\sc W.~Bao, H.~Garcke, R.~N{\"u}rnberg, and Q.~Zhao}, {\em Volume-preserving
  parametric finite element methods for axisymmetric geometric evolution
  equations}, Journal of Computational Physics, 460 (2022), p.~111180.

\bibitem{bao2021symmetrized}
{\sc W.~Bao, W.~Jiang, and Y.~Li}, {\em A symmetrized parametric finite element
  method for anisotropic surface diffusion of closed curves via a cahn-hoffman
  $\xi$-vector formulation}, arXiv preprint arXiv:2112.00508,  (2021).

\bibitem{bao2017parametric}
{\sc W.~Bao, W.~Jiang, Y.~Wang, and Q.~Zhao}, {\em A parametric finite element
  method for solid-state dewetting problems with anisotropic surface energies},
  J. Comput. Phys., 330 (2017), pp.~380--400.

\bibitem{bao2021structure}
{\sc W.~Bao and Q.~Zhao}, {\em A structure-preserving parametric finite element
  method for surface diffusion}, SIAM Journal on Numerical Analysis, 59 (2021),
  pp.~2775--2799.

\bibitem{bao2021structurepreserving}
{\sc W.~Bao and Q.~Zhao}, {\em A structure-preserving parametric finite element
  method for surface diffusion}, SIAM J. Numer. Anal., 59 (2021),
  pp.~2775--2799.

\bibitem{barrett2007parametric}
{\sc J.~W. Barrett, H.~Garcke, and R.~N{\"u}rnberg}, {\em A parametric finite
  element method for fourth order geometric evolution equations}, J. Comput.
  Phys., 222 (2007), pp.~441--467.

\bibitem{barrett2008numerical}
{\sc J.~W. Barrett, H.~Garcke, and R.~N{\"u}rnberg}, {\em Numerical
  approximation of anisotropic geometric evolution equations in the plane}, IMA
  J. Numer. Anal., 28 (2008), pp.~292--330.

\bibitem{barrett2008parametric}
{\sc J.~W. Barrett, H.~Garcke, and R.~N{\"u}rnberg}, {\em On the parametric
  finite element approximation of evolving hypersurfaces in r3}, Journal of
  Computational Physics, 227 (2008), pp.~4281--4307.

\bibitem{barrett2008variational}
{\sc J.~W. Barrett, H.~Garcke, and R.~N{\"u}rnberg}, {\em A variational
  formulation of anisotropic geometric evolution equations in higher
  dimensions}, Numerische Mathematik, 109 (2008), pp.~1--44.

\bibitem{barrett2010stable}
{\sc J.~W. Barrett, H.~Garcke, and R.~N{\"u}rnberg}, {\em On stable parametric
  finite element methods for the stefan problem and the mullins--sekerka
  problem with applications to dendritic growth}, Journal of Computational
  Physics, 229 (2010), pp.~6270--6299.

\bibitem{barrett2012numerical}
{\sc J.~W. Barrett, H.~Garcke, and R.~N{\"u}rnberg}, {\em Numerical
  computations of faceted pattern formation in snow crystal growth}, Physical
  Review E, 86 (2012), p.~011604.

\bibitem{barrett2012parametric}
{\sc J.~W. Barrett, H.~Garcke, and R.~N{\"u}rnberg}, {\em Parametric
  approximation of isotropic and anisotropic elastic flow for closed and open
  curves}, Numerische Mathematik, 120 (2012), pp.~489--542.

\bibitem{bekhtereva1988indium}
{\sc O.~Bekhtereva, Y.~Gavrilyuk, V.~Lifshits, and B.~Churusov}, {\em Indium
  surface phase formation on si (111) surface and their role in diffusion and
  desorption. poverkhnost’}, Physika, Khimiia i Mekhanika, 8 (1988), p.~54.

\bibitem{cahn1974vector}
{\sc J.~Cahn and D.~Hoffman}, {\em A vector thermodynamics for anisotropic
  surfaces: I. curved and faceted surfaces}, The Selected Works of John W.
  Cahn,  (1998), pp.~315--324.

\bibitem{Cahn94}
{\sc J.~W. Cahn and J.~E. Taylor}, {\em Overview no. 113 surface motion by
  surface diffusion}, Acta Metall. Mater., 42 (1994), pp.~1045--1063.

\bibitem{chang1999thermodynamic}
{\sc L.-S. Chang, E.~Rabkin, B.~Straumal, B.~Baretzky, and W.~Gust}, {\em
  Thermodynamic aspects of the grain boundary segregation in cu (bi) alloys},
  Acta Materialia, 47 (1999), pp.~4041--4046.

\bibitem{clarenz2000anisotropic}
{\sc U.~Clarenz, U.~Diewald, and M.~Rumpf}, {\em Anisotropic geometric
  diffusion in surface processing}, IEEE Visualization 2000, 2000.

\bibitem{deckelnick2005computation}
{\sc K.~Deckelnick, G.~Dziuk, and C.~M. Elliott}, {\em Computation of geometric
  partial differential equations and mean curvature flow}, Acta Numer., 14
  (2005), pp.~139--232.

\bibitem{du2010tangent}
{\sc P.~Du, M.~Khenner, and H.~Wong}, {\em A tangent-plane marker-particle
  method for the computation of three-dimensional solid surfaces evolving by
  surface diffusion on a substrate}, J. Comput. Phys., 229 (2010),
  pp.~813--827.

\bibitem{Fonseca14}
{\sc I.~Fonseca, A.~Pratelli, and B.~Zwicknagl}, {\em Shapes of epitaxially
  grown quantum dots}, Arch. Ration. Mech. Anal., 214 (2014), pp.~359--401.

\bibitem{giga2006surface}
{\sc Y.~Giga}, {\em Surface evolution equations}, Springer, 2006.

\bibitem{gurtin2002interface}
{\sc M.~E. Gurtin and M.~E. Jabbour}, {\em Interface evolution in three
  dimensions with curvature-dependent energy and surface diffusion:
  Interface-controlled evolution, phase transitions, epitaxial growth of
  elastic films}, Archive for rational mechanics and analysis, 163 (2002),
  pp.~171--208.

\bibitem{hauffe1955application}
{\sc K.~Hauffe}, {\em The application of the theory of semiconductors to
  problems of heterogeneous catalysis}, in Advances in Catalysis, vol.~7,
  Elsevier, 1955, pp.~213--257.

\bibitem{Hausser07}
{\sc F.~Hau{\ss}er and A.~Voigt}, {\em A discrete scheme for parametric
  anisotropic surface diffusion}, J. Sci. Comput., 30 (2007), pp.~223--235.

\bibitem{hoffman1972vector}
{\sc D.~W. Hoffman and J.~W. Cahn}, {\em A vector thermodynamics for
  anisotropic surfaces: I. fundamentals and application to plane surface
  junctions}, Surface Science, 31 (1972), pp.~368--388.

\bibitem{jiang2012}
{\sc W.~Jiang, W.~Bao, C.~V. Thompson, and D.~J. Srolovitz}, {\em Phase field
  approach for simulating solid-state dewetting problems}, Acta Mater., 60
  (2012), pp.~5578--5592.

\bibitem{jiang2016solid}
{\sc W.~Jiang, Y.~Wang, Q.~Zhao, D.~J. Srolovitz, and W.~Bao}, {\em Solid-state
  dewetting and island morphologies in strongly anisotropic materials}, Scr.
  Mater., 115 (2016), pp.~123--127.

\bibitem{jiang2019sharp}
{\sc W.~Jiang and Q.~Zhao}, {\em Sharp-interface approach for simulating
  solid-state dewetting in two dimensions: A {Cahn--Hoffman} $\xi$-vector
  formulation}, Phys. D, 390 (2019), pp.~69--83.

\bibitem{li2020energy}
{\sc Y.~Li and W.~Bao}, {\em An energy-stable parametric finite element method
  for anisotropic surface diffusion}, J. Comput. Phys., 446 (2021), p.~110658.

\bibitem{li1999numerical}
{\sc Z.~Li, H.~Zhao, and H.~Gao}, {\em A numerical study of electro-migration
  voiding by evolving level set functions on a fixed cartesian grid}, J.
  Comput. Phys., 152 (1999), pp.~281--304.

\bibitem{Mullins57}
{\sc W.~W. Mullins}, {\em Theory of thermal grooving}, J. Appl. Phys., 28
  (1957), pp.~333--339.

\bibitem{Naffouti17}
{\sc M.~Naffouti, R.~Backofen, M.~Salvalaglio, T.~Bottein, M.~Lodari, A.~Voigt,
  T.~David, A.~Benkouider, I.~Fraj, L.~Favre, et~al.}, {\em Complex dewetting
  scenarios of ultrathin silicon films for large-scale nanoarchitectures}, Sci.
  Advances, 3 (2017), p.~1472.

\bibitem{CFDTool}
{\sc P.~Simulation}, {\em Cfdtool - matlab cfd simulation gui \& toolbox,
  github.}
\newblock
  \url{https://github.com/precise-simulation/cfdtool/releases/tag/1.8.3}, 2022.

\bibitem{taylor1992ii}
{\sc J.~E. Taylor}, {\em Mean curvature and weighted mean curvature}, Acta
  Metall. Mater., 40 (1992), pp.~1475--1485.

\bibitem{taylor1994linking}
{\sc J.~E. Taylor and J.~W. Cahn}, {\em Linking anisotropic sharp and diffuse
  surface motion laws via gradient flows}, J. Stat. Phys., 77 (1994),
  pp.~183--197.

\bibitem{taylor1992overview}
{\sc J.~E. Taylor, J.~W. Cahn, and C.~A. Handwerker}, {\em Overview no. 98
  i—geometric models of crystal growth}, Acta Metallurgica et Materialia, 40
  (1992), pp.~1443--1474.

\bibitem{Thompson12}
{\sc C.~V. Thompson}, {\em Solid-state dewetting of thin films}, Annu. Rev.
  Mater. Res., 42 (2012), pp.~399--434.

\bibitem{wang2015sharp}
{\sc Y.~Wang, W.~Jiang, W.~Bao, and D.~J. Srolovitz}, {\em Sharp interface
  model for solid-state dewetting problems with weakly anisotropic surface
  energies}, Phys. Rev. B, 91 (2015), p.~045303.

\bibitem{wheeler1999cahn}
{\sc A.~Wheeler}, {\em {Cahn--Hoffman} $\xi$-vector and its relation to diffuse
  interface models of phase transitions}, J. Stat. Phys., 95 (1999),
  pp.~1245--1280.

\bibitem{Suo97}
{\sc L.~Xia, A.~F. Bower, Z.~Suo, and C.~Shih}, {\em A finite element analysis
  of the motion and evolution of voids due to strain and electromigration
  induced surface diffusion}, J. Mech. Phys. Solids, 45 (1997), pp.~1473--1493.

\bibitem{xu2009local}
{\sc Y.~Xu and C.-W. Shu}, {\em Local discontinuous {Galerkin} method for
  surface diffusion and {Willmore flow} of graphs}, J. Sci. Comput., 40 (2009),
  pp.~375--390.

\bibitem{Ye10a}
{\sc J.~Ye and C.~V. Thompson}, {\em Mechanisms of complex morphological
  evolution during solid-state dewetting of single-crystal nickel thin films},
  Appl. Phys. Lett., 97 (2010), p.~071904.

\bibitem{Zhao}
{\sc Q.~Zhao, W.~Jiang, and W.~Bao}, {\em A parametric finite element method
  for solid-state dewetting problems in three dimensions}, SIAM J. Sci.
  Comput., 42 (2020), pp.~B327--B352.

\end{thebibliography}
\end{document}